\documentclass[10pt,final,reqno]{amsart}
\usepackage{amsfonts,latexsym,url,amsmath,amssymb,comment,enumitem}
\usepackage{upgreek}
\excludecomment{commentforus}
\usepackage{subfigure}
\usepackage{url,centernot}
\usepackage{graphicx,color}
\usepackage[color,notref,notcite]{showkeys}
\usepackage{bm}
\usepackage{stmaryrd}
\usepackage{textcomp}
\DeclareMathAlphabet\mathbfcal{OMS}{cmsy}{b}{n}

\newcommand{\mathbi}[1]{{\boldsymbol #1}}
\def\N{\mathbb{N}}

\def\R{\mathbb{R}}

\def\err{\mathsf{err}}

\def\stab{\mathfrak{S}}
\def\<{\langle}
\def\>{\rangle}
\def\Poly{\mathbb{P}}

\def\div{{\rm div}}

\def\norm#1#2{\Vert#1\Vert_{#2}}

\newcounter{cst}

\def\bt{\begin{theorem}}
\def\et{\end{theorem}}
\def\bl{\begin{lemma}}
\def\el{\end{lemma}}
\def\bc{\begin{corollary}}
\def\ec{\end{corollary}}
\def\bd{\begin{definition}}
\def\ed{\end{definition}}
\def\br{\begin{remark}}
\def\er{\end{remark}}

\newcommand{\disc}{{\mathcal D}}

\def \hessian{\mathcal H}

\def \hd{\hessian_\disc}
\def \wdspace{H}
\def\cv{K}
\def \symd{\R^{d\times d}}
\def \sym2{\R^{2\times 2}}
\def\cell{K}
\DeclareMathOperator*{\argmin}{argmin}

\newcommand{\mesh}{{\mathcal M}}

\newcommand{\edge}{{\sigma}}

\newcommand{\edges}{{\mathcal F}}              
\newcommand{\edgescv}{{{\edges}_\cv}}  
  
\newcommand{\edgesext}{{{\edges}_{\rm ext}}} 
\newcommand{\edgesint}{{{\edges}_{\rm int}}}

\newcommand{\x}{\mathbi{x}}

\newcommand{\vertices}{{\mathcal V}}

\newcommand{\bu}{\overline{u}}
\newcommand{\bv}{\overline{v}}

\newcommand{\be}{\begin{equation}}
\newcommand{\ee}{\end{equation}}

\renewcommand{\O}{\Omega}

\renewcommand{\d}{{\,\rm d}}

\newcommand{\ba}{\begin{array}{llll}   }
\newcommand{\bac}{\begin{array}{c}}
\newcommand{\bari}{\begin{array}{r}}
\newcommand{\ea}{\end{array}}

\newcommand{\NORM}[1]{{\left\vert\kern-0.25ex\left\vert\kern-0.25ex\left\vert #1 
    \right\vert\kern-0.25ex\right\vert\kern-0.25ex\right\vert}}

\newcommand{\dm}{\disc_m}  
\def\hat#1{\widehat{#1}}
\newcommand{\bsymb}[1]{{\boldsymbol #1}}
\newcommand{\bfX}{\bsymb{X}}
\newcommand{\bfXd}{\bsymb{X}_{\disc,0}}

\usepackage{graphicx}
\usepackage{epstopdf}
\usepackage{pdfpages}
\usepackage{subfig,epsfig}
\usepackage{textcomp}
\usepackage{color}
\usepackage{tikz}
\usepackage{caption}

\usepackage{cancel}

\usetikzlibrary{arrows}
\usetikzlibrary{shapes,calc}

\usepackage[square,sort&compress,comma,numbers]{natbib}

\def\assum#1{{\rm\textbf{(A#1)}}}
\def\propty#1{{\rm\textbf{(P#1)}}}

\newtheorem{theorem}{Theorem}[section]
\newtheorem{remark}[theorem]{Remark}

\newtheorem{lemma}[theorem]{Lemma} 
\newtheorem{definition}[theorem]{Definition}

\newtheorem{corollary}[theorem]{Corollary}

\numberwithin{equation}{section}

\def\XXint#1#2#3{{\setbox0=\hbox{$#1{#2#3}{\int}$ }
\vcenter{\hbox{$#2#3$ }}\kern-.6\wd0}}

\usepackage[normalem]{ulem}
\newcounter{corr}
\definecolor{violet}{rgb}{0.580,0.,0.827}
\newcommand{\corr}[3]{\typeout{Warning : a correction remains in page
\thepage}
				\stepcounter{corr}        
				{\color{blue}\ifmmode\text{\,\sout{\ensuremath{#1}}\,}\else\sout{#1}\fi}
       {\color{red}#2}
       {\color{violet} #3}}

\newcounter{cexp}
\catcode`\@=11
\def\terml#1{T_{\refstepcounter{cexp}\@bsphack
\protected@write\@auxout{}%
           {\string\newlabel{#1}{{\thecexp}{\thepage}}}\thecexp}}
\catcode`\@=12

\begin{document}
\title[HDM for fourth order semi-linear elliptic equations]{Hessian discretisation method for fourth order semi-linear elliptic equations: applications to the von K\'{a}rm\'{a}n and Navier--Stokes models}
\author{J\'erome Droniou}
\address{School of Mathematical Sciences, Monash University, Clayton,
	Victoria 3800, Australia.
	\texttt{jerome.droniou@monash.edu}}
\author{Neela Nataraj}
\address{Department of Mathematics, Indian Institute of Technology Bombay, Powai, Mumbai 400076, India.
	\texttt{neela@math.iitb.ac.in}}
\author{Devika Shylaja}
\address{IITB-Monash Research Academy, Indian Institute of Technology Bombay, Powai, Mumbai 400076, India.
	\texttt{devikas@math.iitb.ac.in}}
\maketitle
\begin{abstract} 
This paper deals with the Hessian discretisation method (HDM) for fourth order semi-linear elliptic equations with a trilinear nonlinearity. The HDM provides a generic framework for the convergence analysis of several numerical methods, such as, the conforming and non-conforming finite element methods (ncFEMs) and methods based on gradient recovery (GR) operators. The Adini ncFEM and GR method, a specific scheme that is based on cheap, local reconstructions of higher-order derivatives from piecewise linear functions, are analysed for the first time for fourth order semi-linear elliptic equations with trilinear nonlinearity. Four properties namely, the coercivity, consistency, limit-conformity and compactness enable the convergence analysis in HDM framework that does not require any regularity of the exact solution. Two important problems in applications namely, the Navier--Stokes equations in stream function vorticity formulation and the von K\'{a}rm\'{a}n equations of plate bending are discussed. Results of numerical experiments are presented for the Morley ncFEM and GR method.
\end{abstract}


\section{Introduction}
Fourth order linear and non-linear elliptic problems arise in a wide range of application areas that include thin plate theories of elasticity, thin beams and shells. In this paper, we study the approximation of fourth order semi-linear problems with a trilinear nonlinearity and clamped boundary conditions in an abstract setting using the Hessian Discretisation Method (HDM). 

\smallskip

The HDM for fourth order semi-linear equations with a trilinear nonlinearity is a unified framework for the convergence analysis of several numerical methods, such as, the conforming finite element methods (FEMs), Adini and Morley non-conforming finite element methods (ncFEMs), and methods based on gradient recovery (GR) operators \cite{ciarlet1978finite,BL_FEM,BL_MFEM}. The framework of HDM is based on a quadruplet, referred to as a Hessian discretisation (HD), that involves a discrete space, a reconstructed function, a reconstructed gradient and a reconstructed Hessian. The Hessian schemes (HS) are discrete versions of the weak formulation obtained by replacing the continuous space, function, gradient and Hessian by corresponding discrete ones. The HDM allows a complete convergence analysis for families of numerical methods through a small number of properties: coercivity, consistency, limit-conformity and compactness.

\smallskip
%

\smallskip   
 The abstract problem considered in this article applies in particular to the stream function vorticity formulation of the incompressible 2D Navier--Stokes problem \cite{BRR,NS_FE79,CCGMNN2020} and to the von K\'{a}rm\'{a}n equations \cite{ciarlet_plate}. There are advantages in using the stream function vorticity formulation of the incompressible Navier--Stokes equations to compute 2D flows: the continuity equation is automatically satisfied, only one vorticity (transport) equation has to be solved, the streamlines of the flow are given by level curves of the stream function, and the vorticity is a conserved quantity. The von K\'{a}rm\'{a}n equations is a system of fourth order semi-linear elliptic equations that describes the bending of very thin elastic plates. The numerical analysis of von K\'{a}rm\'{a}n equations has been studied using conforming FEMs in \cite{Brezzi,ng1}, Morley ncFEM in \cite{ng2,CCGMNN2020}, mixed FEMs \cite{BRR,TM76}, a $C^0$ interior penalty method in \cite{SBMNARLS} and a discontinuous Galerkin method in \cite{CCGMNN18}. To the best of our knowledge, {\it the Adini ncFEM and the method based on GR operator have not been studied in literature for fourth order non-linear elliptic equations}. The Adini ncFEM and the GR method are analysed in the HDM framework for the fourth order linear equations \cite{HDM_linear, DS_HDM} along with the conforming FEMs, Morley ncFEM and finite volume methods.

\smallskip
 The analysis via error estimate has been considered in literature, for conforming, nonconforming, discontinuous Galerkin FEMs and mixed FEM for the Navier--Stokes equations and von K\'{a}rm\'{a}n equations under the assumption that $(i)$ the exact solution has extra regularity and $(ii)$ the linearised problem around the exact solution is well-posed. In this article, a different approach is employed for the convergence analysis using the four properties associated with the HD. {\it The convergence analysis is based on compactness techniques approach that does not rely on any smoothness or structural assumption on the continuous solution.} In this approach, the solution to the weak formulation is obtained as the limit of a sequence of solutions to the approximate problem; the existence of solution for the continuous model is therefore established as a consequence of this convergence analysis. {\it  To the best of our knowledge, this is the first time that this approach is considered for the numerical analysis of von K\'{a}rm\'{a}n equations.} 

\smallskip
The contributions of this article are the following:
\begin{itemize}
		\item {\it Convergence analysis} by compactness techniques for an abstract semi-linear fourth order model, without any extra-regularity assumption on the exact solution. This analysis employs only four properties, namely, the  coercivity, consistency,  limit-conformity and compactness. 
		
		\item Design and analysis of {\it the Adini ncFEM} and {\it GR method} for fourth order semi-linear elliptic equations.
		
		\item  {\it A unified framework} provided by HDM for fourth order semi-linear elliptic equations with a trilinear nonlinearity, in an abstract set-up that applies to several numerical methods. 
		\item {\it Applications} to the stream function vorticity formulation of 2D Navier--Stokes equation and the von K\'{a}rm\'{a}n equations using the examples of HDM, namely, conforming FEMs, Adini and Morley ncFEMs, and GR methods.
	\item \emph{Numerical experiments} on the approximation of Navier--Stokes equation and von K\'{a}rm\'{a}n equations.
\end{itemize}

\medskip

The paper is organised as follows. The abstract problem, some examples and an illustration of application of the main results are presented in Section \ref{abstract_setting}. Section \ref{sec.HDM} deals with the HDM for fourth order non-linear problems and some examples that fit into the HDM framework. The four properties that are needed for the convergence analysis of HDM are described in Section \ref{sec.properties}. These properties are verified for several numerical methods. The main result of the paper using  compactness techniques approach for the convergence analysis in the abstract framework is presented in Sections \ref{con.reg}. Section \ref{sec.numericalresults} provides the results of numerical experiments for the method based on GR operators and Morley ncFEM. A section on conclusion (Section \ref{sec:conclusion}) and the proof of the properties for ncFEMs and GR methods (Appendix \ref{appen.prop}) complete the paper.

\smallskip
\textbf{Notations}. Let $\O \subset \R^d$ $ (d\ge1)$ be a bounded domain  with boundary $\partial \Omega$ and let the outer normal be denoted by $n$. For brevity, we follow the Einstein summation convention that implies summation over a set of indexed terms in a product of vectors, tensors or differential operators unless otherwise stated. The scalar product on $\symd$ is defined by $\xi:\phi=\xi_{ij}\phi_{ij}$. For a function $\xi : \O \rightarrow \symd$, denoting the Hessian operator by $\hessian$, set $\hessian : \xi = \partial _{ij}\xi_{ij}$. For $a,b\in\R^d$, let $a \otimes b$ denotes the 2-tensor with coefficients $a_ib_j$. The Lebesgue measure of a measurable set $E\subset \R^d$ (resp. the set of all matrices in $L^2(\O)^{d \times d}$) is denoted by $|E|$ (resp. $L^2(\O;\symd)$). The standard $L^2$ inner product and norm (applied on $L^2(\O)$, $L^2(\O;\R^d)$, and $L^2(\O;\symd)$) are denoted by $(\cdot,\cdot)$ and $\|{\cdot}\|$. For $r>0$, let $\norm{{\cdot}}{L^4}$ denotes the norm in $L^4(\O)^r$.

\section{Model problem and application of the main results}\label{abstract_setting}
We present here the abstract setting of weak formulation of semi-linear fourth order elliptic problems with a trilinear nonlinearity and clamped boundary conditions. An example of application of the main result is also stated at the end of this section. 

\medskip

Let $k\ge 1$ be an integer and, for $E$ a vector space, set $\bsymb{E}=E^k$. For simplicity of notation, the norm in $\bsymb{E}$ is denoted by $\| \cdot \|_{E}$. Letting $X:=H^2_0(\O)$, the continuous abstract problem seeks $\Psi \in \bfX$ such that
\be \label{abstract_weak}
\mathcal{A}(\hessian \Psi,\hessian \Phi)+\mathcal{B}(\hessian \Psi, \nabla \Psi, \nabla \Phi)=\mathcal{L}(\Phi) \quad \forall \,\Phi \in \bfX,
\ee
where $\hessian \Psi$ and $\nabla \Psi$ are to be understood component-wise: for $\Psi=(\psi_1,\cdots,\psi_k)$, $\hessian \Psi=(\hessian \psi_1,\cdots,\hessian \psi_k)$ and $\nabla \Psi=(\nabla \psi_1,\cdots,\nabla \psi_k)$. Let the following assumptions hold:
\begin{itemize}
	\item[\assum{1}] $\mathcal{A}(\cdot,\cdot)$ is a  continuous and coercive bilinear form on $\bsymb{L}^2(\O;\symd) \times \bsymb{L}^2(\O;\symd)$.
	\item[\assum{2}] $\mathcal{B}(\cdot,\cdot,\cdot)$ is a continuous trilinear form on $\bsymb{L}^2(\O;\symd) \times \bsymb{L}^4(\O;\R^d) \times \bsymb{L}^4(\O;\R^d)$.
	\item[\assum{3}] $\mathcal{B}(\Xi,\Theta,\Theta)=0$ for all $ \Xi \in \bsymb{L}^2(\O;\symd)$ and $\Theta \in \bsymb{L}^4(\O;\R^d)$.
	\item[\assum{4}] $\mathcal{L}(\cdot)$ is a continuous linear form on $\bsymb{L}^2(\O)$.
\end{itemize}

\subsection{Examples}\label{examples}
We show here that the abstract formulation \eqref{abstract_weak} covers the stream function vorticity formulation of the incompressible 2D Navier--Stokes problem and von K\'{a}rm\'{a}n equations.

\subsubsection{Navier--Stokes problem \cite{Lions_NS,BRR}:} For given $f \in L^2(\O)$ where $\Omega \subsetneq \R^2$ and viscosity $\nu >0$, let $u$ solve
\begin{subequations}\label{NS_eq}
\begin{align} 
\nu \Delta^2u + \frac{\partial}{\partial x_1}\bigg((-\Delta u)\frac{\partial u}{\partial x_2}\bigg)&- \frac{\partial}{\partial x_2}\bigg((-\Delta u)\frac{\partial u}{\partial x_1}\bigg)=f \mbox{ in } \O\\
&u=\frac{\partial u}{\partial n}=0 \mbox{ on } \partial\O.
\end{align}
\end{subequations}
Here, the biharmonic operator $\Delta^2$ is defined by $\Delta^2\phi=\phi_{xxxx}+\phi_{yyyy}+2\phi_{xxyy}$. The weak formulation to \eqref{NS_eq} seeks $u \in H^2_0(\O)$ such that
\be \label{NS_weak}
\mathcal{A}(\hessian u,\hessian v)+\mathcal{B}(\hessian u, \nabla u, \nabla v)=\mathcal{L}(v) \quad \forall  v \in H^2_0(\O),
\ee
where for all $\xi,$ $\chi \in L^2(\O;\sym2)$ and $\phi,$ $ \theta \in L^2(\O,\R^2),$
$$\mathcal{A}(\xi,\chi)=\nu \int_{\O}^{}\xi: \chi \d \x, \quad \mathcal{B}(\xi, \phi,\theta)=\int_{\O}^{}\mbox{tr}(\xi) \phi\cdot \mbox{rot}_{\pi/2}(\theta)\d\x,\quad\mathcal{L}(v)=\int_{\O}^{}fv\d\x. $$
Note that $\mbox{tr}(\xi)$ means the trace of the matrix $\xi$ and, for $\theta=(\theta_1,\theta_2)$,
$\mbox{rot}_{\pi/2}(\theta)=\big(-\theta_2, \theta_1\big)^t.$
It is easy to check that $\mathcal{A}(\cdot,\cdot)$, $\mathcal{B}(\cdot,\cdot,\cdot)$ and $\mathcal{L}(\cdot)$ satisfy \assum{1}-\assum{4} with $k=1$. The continuity of $\mathcal{B}(\cdot,\cdot,\cdot)$ follows using the generalized H\"{o}lder's inequality given by $\mathcal{B}(\xi, \phi,\theta) \le \norm{\xi}{}\norm{\phi}{L^4}\norm{\theta}{L^4}$. 

\subsubsection{The von K\'{a}rm\'{a}n equations \cite{ciarlet_plate}: } Given $f \in L^2(\Omega)$ where $\Omega \subsetneq \R^2$, seek the vertical displacement  $u$ and the  Airy stress function $v$ such that  
\begin{subequations}\label{vk_eq}
	\begin{align}%
	\Delta^2u&=[u,v]+f \mbox{ in }\O, \label{vk_eq1}\\
	\Delta^2v&=-\frac{1}{2}[u,u] \mbox{ in }\O,\label{vk_eq2}
	\end{align}
\end{subequations}
with clamped boundary conditions
\be \label{clamped_bc}
u=\frac{\partial u}{\partial n}=v=\frac{\partial v}{\partial n}=0 \mbox{ on } \partial\O.
\ee
The von K\'{a}rm\'{a}n bracket $[\cdot,\cdot]$ is defined by
$[\xi,\chi]=$$\xi_{xx}\chi_{yy}+\xi_{yy}\chi_{xx}-2\xi_{xy}\chi_{xy}$$=\mbox{cof}(\hessian\xi):\hessian\chi$, where $\mbox{cof}(\hessian \xi)$ denotes the co-factor matrix of $\hessian \xi$. Then a weak formulation corresponding to \eqref{vk_eq} seeks $u, v \in H^2_0(\O)$ such that
\begin{subequations}\label{vk_weak}
	\begin{align}
	a(u,\phi_1)+2b(u,\phi_1,v)&=(f,\phi_1) \quad\forall \phi_1 \in H^2_0(\O), \label{vk_weak1} \\
	2a(v,\phi_2)-2b(u,u,\phi_2)&=0 \quad\forall  \phi_2 \in H^2_0(\O), \label{vk_weak2}
	\end{align}
\end{subequations}
where for all $\xi, \chi, \phi \in H^2_0(\O),$
\[
\begin{aligned}
&a(\xi, \chi):=\int_\O \hessian \xi : \hessian \chi\d\x,\;b(\xi,\chi,\phi):=\frac{1}{2}\int_\O\mbox{cof}(\hessian \xi)\nabla\chi \cdot\nabla \phi\d\x=-\frac{1}{2}\int_\O[\xi,\chi]\phi\d\x.
\end{aligned}
\]
Note that $b(\cdot,\cdot,\cdot)$ is derived using the divergence-free rows property \cite{evans_pde} and is symmetric with respect to all variables.  
Summing together \eqref{vk_weak1} and \eqref{vk_weak2}, we obtain an equivalent formulation
in the vector form \eqref{abstract_weak} with $k=2$ that seeks $\Psi=(u,v) \in \bfX$ such that 
\be \label{vk_weak_vector}
\mathcal{A}(\hessian \Psi,\hessian \Phi)+\mathcal{B}(\hessian \Psi,\nabla \Psi, \nabla \Phi)=\mathcal{L}(\Phi) \quad \forall \Phi \in \bfX,
\ee
where for all $\Phi=(\phi_1,\phi_2)$, $\Lambda=(\lambda_1,\lambda_2),\,\Gamma=(\gamma_1,\gamma_2),\, \Theta=(\theta_1,\theta_2)$ and $\Xi=(\xi_1, \xi_2)$
with $\Lambda, \Gamma \in 
\bsymb{L}^2(\O;\sym2)$ and $\Xi,$ $\Theta \in \bsymb{L}^2(\O,\R^2),$
\begin{subequations}\label{vk_weak_ABL}
	\begin{align}
	\mathcal{A}(\Lambda, \Gamma):&=\int_\O \lambda_1 :\gamma_1\d\x +2\int_\O  \lambda_2 : \gamma_2\d\x, \label{vk_weak_A}\\
	\mathcal{B}(\Lambda,\Xi,\Theta):&=\int_\O\mbox{cof}(\lambda_1)\theta_1\cdot\xi_2\d\x- \int_\O\mbox{cof}(\lambda_1)\xi_1 \cdot\theta_2\d\x,\mbox{ and} \label{vk_weak_B}\\
	\mathcal{L}(\Phi):&=(f,\phi_1) \label{vk_weak_L}.
	\end{align}
\end{subequations}
The assumptions \assum{1}-\assum{4} are easy to verify for this example.
 
 \begin{remark}\label{remark.vk.wf}
The more commonly used equivalent weak formulation of the von K\'{a}rm\'{a}n model \cite{Brezzi,ng1,ng2} \eqref{vk_eq} uses a different formulation of the non-linearity in \eqref{vk_weak1}. Specifically, it seeks $(u,v) \in \bfX$ such that 
\begin{subequations}\label{vk_usualweak}
	\begin{align}
	a(u,\phi_1)+2b(u,v,\phi_1)&=(f,\phi_1) \quad\forall \phi_1 \in H^2_0(\O), \label{vk_usualweak1} \\
	a(v,\phi_2)-b(u,u,\phi_2)&=0 \quad\forall \phi_2 \in H^2_0(\O) \label{vk_usualweak2}.
	\end{align}
\end{subequations}
An advantage of \eqref{vk_weak} is that it ensures the proper cancellation in the trilinear term, in a purely algebraic way (corresponding to \assum{3}) without further integration-by-parts. As a consequence, this cancellation, which is at the core of \emph{a priori} estimates on the solution, directly transfers to the discrete level thus avoiding integration-by-parts over $\Omega$. This formulation of the non-linear term is similar in spirit to what is usually done for finite element discretisations of the Navier--Stokes equations, see \cite{temam_NS}. 
\end{remark}
\subsection{Example of application of the main results}
Application of the main results of this article for the Adini ncFEM are stated in this section in a simplified way. Let $\Omega$ be a polygonal domain and $\mesh$ be a conforming mesh \cite{gdm} of rectangles. For any cell $\cell\in\mesh$, let $h_\cell$ denotes the diameter of $\cell$ and let $h:=\max_{K\in\mesh}h_K$ be the mesh-size.

\medskip

The Adini ncFEM corresponding to \eqref{abstract_weak} seeks $\Psi_h \in V_h^k$ such that
\begin{align}\label{Adini.weak}
&\mathcal{A}(\hessian_\mesh \Psi_h,\hessian_\mesh\Phi_h)+\mathcal{B}(\hessian_\mesh \Psi_h, \nabla \Psi_h, \nabla \Phi_h)=\mathcal{L}(\Phi_h) \quad \forall \Phi_h \in V_h^k,
\end{align}
where $V_h$ is the nonconforming Adini finite element space and $\hessian_\mesh \Phi_h$ is the broken Hessian of $\Phi_h$ (see Section \ref{sec:ncfem} for details). Note that $\hessian_\mesh \Phi_h$ and $\nabla \Phi_h$ act component-wise whenever it is applied to a vector-valued function.

\medskip

The next result states a convergence result that does not require any smoothness of the exact solution or the assumption that the exact solution is regular, that is, the linearized problem around the exact solution is well-posed for the Adini ncFEM. The result is stated in Theorem \ref{convergence_abstract_non-linear} in a more generic way in the HDM framework which includes conforming FEMs, Adini and Morley ncFEMs, and the GR methods.

\smallskip

\begin{theorem}[Convergence of Adini FEM for von K\'{a}rm\'{a}n and Navier--Stokes equations]
	Let the assumptions $\assum{1}-\assum{4}$ hold true. Then there exists at least one solution $\Psi_h$ to \eqref{Adini.weak}. Moreover, as $h \rightarrow 0$, up to a subsequence of $\mesh_h$, there exists a solution $\Psi$ of the abstract problem \eqref{abstract_weak} such that $\Psi_h \rightarrow \Psi$ in $L^2(\O)^k$, $\nabla \Psi_h \rightarrow \nabla \Psi$ in $L^4(\O;\R^d)^k$, and $\hessian_{\mesh_h} \Psi_h \rightarrow \hessian \Psi$ in $L^2(\O; \R^{d \times d})^k$.
\end{theorem}

%
%
\section{The Hessian discretisation method} \label{sec.HDM}
	This section is devoted to the presentation of the HDM for fourth order non-linear elliptic equations, the design of which is adapted from the HDM for linear problems (see Remark \ref{remark.linear}). 
	
	\begin{definition}[Hessian discretisation]\label{HD}~
	A Hessian discretisation for fourth order non-linear elliptic equations with clamped boundary conditions is a quadruplet $\disc=(X_{\disc,0},\Pi_\disc,\nabla_\disc,\hd)$ such that
	\begin{itemize}
		\item  $X_{\disc,0}$ is a finite dimensional real vector space, 
		\item the linear mapping $\Pi_\disc:X_{\disc,0}  \rightarrow L^2(\O)$ gives a reconstructed discrete function in $L^2(\O)$ from vectors in $X_{\disc,0},$
		\item the linear mapping $\nabla_\disc:X_{\disc,0}  \rightarrow L^4(\O;\R^d)$ gives a reconstructed discrete gradient in $L^4(\O;\R^d)$ from vectors in  $X_{\disc,0}$,
		\item the linear mapping $\hd:X_{\disc,0}  \rightarrow L^2(\O;\R^{d \times d})$ gives a reconstructed discrete version of Hessian in $L^2(\O;\R^{d \times d})$ from $X_{\disc,0}$. The operator $\hd$ is such that $\norm{{\cdot}}{\disc}=:\norm{\hd \cdot}{}$ is a norm on $X_{\disc,0}.$
	\end{itemize}
\end{definition}
 In order to approximate \eqref{abstract_weak} by the HDM, consider a HD $\disc=(X_{\disc,0},\Pi_\disc,\nabla_\disc,\hd)$ in the sense of Definition \ref{HD}. The associated HS for \eqref{abstract_weak} seeks $\Psi_\disc \in \bfXd:=X_{\disc,0}^k$ such that
\begin{align}
&\mathcal{A}(\hd \Psi_\disc,\hd \Phi_\disc)+\mathcal{B}(\hd \Psi_\disc, \nabla_\disc \Psi_\disc, \nabla_\disc \Phi_\disc)=\mathcal{L}(\Pi_\disc\Phi_\disc) \quad \forall \Phi_\disc \in \bfXd,\label{abstract_HS}
\end{align}
where $\hd \Phi_\disc$, $\nabla_\disc \Phi_\disc$ and $\Pi_\disc \Phi_\disc$ act component-wise in the sense that if $\Phi_\disc=(\phi_{\disc,1}, \cdots,\phi_{\disc,k})$ and $F_\disc\in\{\Pi_\disc,\nabla_\disc,\hd\}$, then $F_\disc \Phi_\disc=(F_\disc\phi_{\disc,1},\cdots,F_\disc\phi_{\disc,k})$. 

\subsection{Examples of HDMs}\label{sec:class.in.HDM}

This section deals with examples of schemes that fit into the HDM framework. Let us begin with mesh notation.

\smallskip

Let $\mesh$ be a conforming mesh \cite{gdm} of triangles or rectangles (that depends on the method under consideration). For any $\cell\in\mesh$, the center of mass of $K$ is defined by $\overline{\x}_\cell$, $|\cell|>0$ denote the measure of $\cell$,  and $n_K$ be the outer unit normal to $K$. Let $\edges$ be the set of all edges of the mesh and the measure of $\edge\in\edges$ be denoted by $|\edge|$. Let the set of vertices in $\mesh$ be denoted  by $\vertices$. Let $\mathcal V_{\rm int}$ (resp. $\mathcal V_{\rm ext}$) denote the set of internal vertices of $\mesh$ (resp. vertices on $\partial \O$). The meshes are assumed to be regular \cite{ciarlet1978finite} in the classical sense that the ratio of the diameter and the radius of the largest ball centered at $\overline{\x}_\cell$ and included in $\cell$ is uniformly bounded with a bound independent of $h$. The notation ``$A\lesssim B$'' means that there exists a generic constant $C$ independent of the mesh parameter $h$ such that $A \le CB$.

Let  $\ell\ge 0$ be an integer and $\cell \in \mesh$. Let the space of polynomials of degree at most $\ell$ in $\cell$ be denoted by $\mathbb{P}_\ell(\cell)$ and let $\mathbb{P}_\ell(\mesh)$ be the broken polynomial space.
\subsubsection{\bf Conforming FEM}\label{sec:conf}
The finite element space $V_h$ is a subspace of $H^2_0(\O)$. A HD is defined by $X_{\disc,0}=:V_h$ and, for $v\in X_{\disc,0}$, $\Pi_\disc v=v$, $\nabla_\disc v=\nabla v$ and $\hd v=\hessian v$. 
Classical $C^1$ elements that are used for the approximation the solution of fourth order elliptic problems are the Argyris and Bogner--Fox--Schmit finite elements, see \cite{ciarlet1978finite} for more details.

\subsubsection{\bf Non-conforming FEM}\label{sec:ncfem}

We show here that two ncFEMs in dimension $d=2$, namely the Morley FEM and the Adini FEM, fit into the framework of HDM. 

\smallskip

\textsc{(i) The Morley element \cite{ciarlet1978finite}:} 
\begin{figure}
	\begin{center}
		\begin{minipage}[b]{0.35\linewidth}
			{\includegraphics[width=4.5cm]{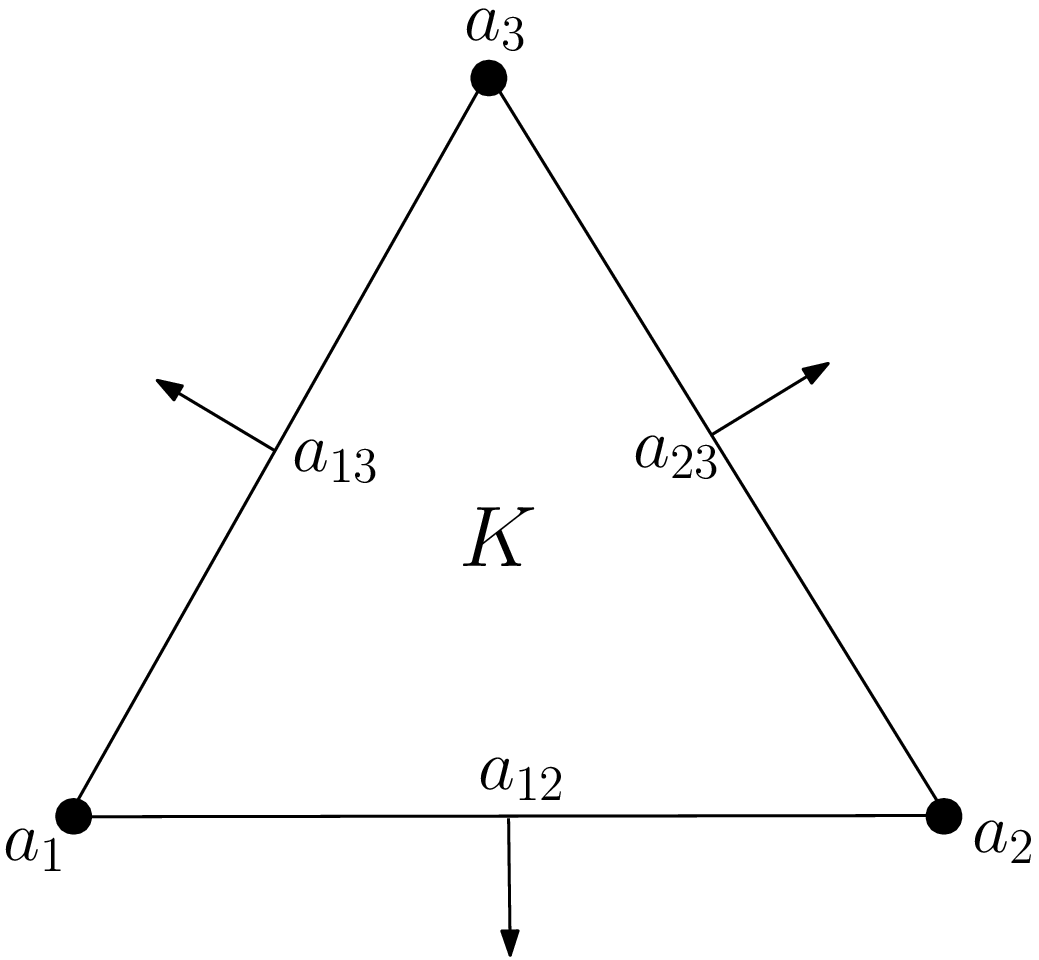}}
		\end{minipage}
		\qquad \quad
		\begin{minipage}[b]{0.4\linewidth}
			{\includegraphics[width=5.5cm]{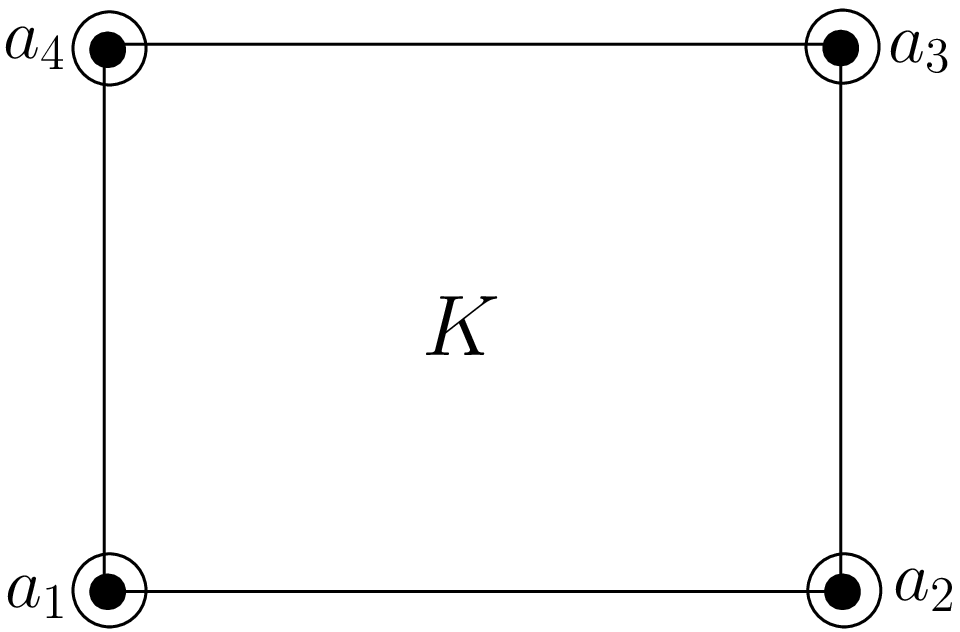}}
		\end{minipage}
		\caption{Morley element (left) and Adini element (right)}\label{fig:MorleyAdiniElement}
	\end{center}
\end{figure}
Let $
 \llbracket \phi \rrbracket$ be the jump of $\phi$ across the edges. The nonconforming Morley element space associated with the mesh $\mesh$ is defined by
\begin{align*}
V_h & =: \big\{\phi \in \Poly_2(\mesh) |\phi \mbox{ is continuous at } \mathcal V_{\rm int}\mbox{ and vanishes at }\mathcal V_{\rm ext},\, \\ &\qquad\forall \edge \in \edgesint, \, \int_{\edge}^{}\bigg\llbracket\frac{\partial \phi}{\partial n}\bigg\rrbracket ds=0;\,\forall \edge \in \edgesext, \, \int_{\edge}^{}\frac{\partial \phi}{\partial n}ds=0\big\}.
\end{align*}
On each triangle, the local degrees of freedom are the values of the function at each vertex and the values of the normal derivatives at the midpoints of edges. See Figure \ref{fig:MorleyAdiniElement} (left) for an illustration.

\begin{definition}[HD for the Morley triangle]\label{def.HD.Morley}
	Each $v_\disc\in X_{\disc,0}$  is a vector of degrees of freedom at the vertices of the mesh (with zero values at boundary vertices) and at the midpoint of the edges opposite to these vertices (with zero values at midpoint of the boundary edges). The function $\Pi_\disc v_\disc$ is such that $(\Pi_\disc v_\disc)_{|K}\in \mathbb{P}_2(\cell)$ and $\Pi_\disc v_\disc$ (resp. its normal derivatives) takes the values at the vertices (resp.\ at the edge midpoints) dictated by $v_\disc$, $\nabla_\disc v_\disc= \nabla_{\mesh}(\Pi_\disc v_\disc)$ is the broken gradient and $\hd v_\disc= \hessian_{\mesh}(\Pi_\disc v_\disc)$ is the broken Hessian.
\end{definition}

\textsc{(ii) The Adini element \cite{ciarlet1978finite}:} 
The Adini finite element space is the subspace of $H^1_0(\O) \cap C^0(\overline{\O})$ defined by
\begin{align*}
V_h  =: \{&v_h \in L^2(\O);\, v_h\rvert_\cell \in \mathbb{P}_\cell \,\forall \,\cell \in \mesh, v_h \mbox{ and }\nabla v_h \mbox{ are continuous }\\ &\mbox{ at the vertices $\vertices$, and vanish at the vertices in }\vertices_{\rm ext}\},
\end{align*}
where $\mathbb{P}_\cell := \Poly_3(K) \oplus \{x_1x^3_2\} \oplus\{x^3_1x_2\}$. The set of degrees of freedom in each cell are the values of function and all first order derivatives at each vertex. This is shown in Figure \ref{fig:MorleyAdiniElement} (right).
\begin{definition}[HD for the Adini rectangle]\label{def.HD.Adini}
	Each $v_\disc\in X_{\disc,0}$ is a vector of three values at each vertex of the mesh (with zero values at boundary vertices), corresponding to function and gradient values, $\Pi_\disc v_\disc$ is the function such that $(\Pi_\disc v_\disc)_{|K}\in \mathbb{P}_\cell$ and its derivatives take the values at the vertices dictated by $v_\disc$, $\nabla_\disc v_\disc=\nabla(\Pi_\disc v_\disc)$ and $\hd v_\disc= \hessian_{\mesh}(\Pi_\disc v_\disc)$.
\end{definition}

\subsubsection{\bf Method based on GR operators}\label{sec.grmethod}
In this section, we consider the scheme based on a gradient reconstruction using biorthogonal systems inspired by \cite{BL_FEM,HDM_linear}. This method is attractive as the approximation space consists of continuous piecewise linear functions and the discrete Hessian is constructed by using a GR operator.
Let $(V_h,Q_h,I_h,\stab_h)$ be a quadruplet of a finite element space $V_h\subset H^1_0(\O)$, a projector $Q_h:L^2(\O)\to V_h$, an interpolant $I_h:{H^2_0(\O)}\to V_h$
and a function $\stab_h\in L^\infty(\O;\R^d)$ that stabilises the reconstructed Hessian such that,
\begin{itemize}
	\item[\propty{0}] [\emph{Strucure of $V_h$ and $I_h$}] For all $z \in V_h$, $\norm{\nabla z}{}\lesssim h^{-1}\norm{z}{}$
	and, for $\varphi \in H^2_0(\O)$, 
	$\norm{\nabla I_h \varphi-\nabla \varphi}{} \lesssim h \norm{\varphi}{H^2(\O)}$. 
	\item[\propty{1}] [\emph{Stability of $Q_h$}] For $\phi \in L^2(\O)$,  $\norm{Q_h \phi}{}
	\lesssim  \norm{\phi}{}.$ 
	\item[\propty{2}] [\emph{$Q_h\nabla I_h$ approximates $\nabla$}] 
	There exists $W$ densely embedded in $H^3(\Omega) \cap H^2_0(\Omega)$ such that $\norm{Q_h\nabla I_h \psi-\nabla \psi}{} \lesssim h^2\norm{\psi}{W}$ for all $\psi\in W$.
	\item[\propty{3}] [\emph{$H^1$ approximation property of $Q_h$}] For $w \in {H^2(\O)\cap H^1_0(\O)}$, 
	$\norm{\nabla Q_hw-\nabla w}{} \lesssim h\norm{w}{H^2(\O)}.$
	\item[\propty{4}] [\emph{Asymptotic density of $[(Q_h\nabla-\nabla)(V_h)]^\bot$}] Setting $N_h=[(Q_h\nabla-\nabla)(V_h)]^\bot$, where the orthogonality is with respect to  the $L^2(\O;\R^d)$-inner product, the following approximation property holds:
	\[
	\inf_{\mu_h\in N_h}\norm{\mu_h-\varphi}{}\lesssim h\norm{\varphi}{H^1(\O)}\quad \forall  \varphi \in H^1(\O;\R^d).
	\]
	\item[\propty{5}] [\emph{Stabilisation function}] $1\le |\stab_h| \lesssim 1$ and, for all $\cell \in \mesh$, 
	$$\left[\stab_{h|K}\otimes (Q_h\nabla-\nabla)(V_h(\cell))\right] \perp \nabla V_h(\cell)^d,$$
	where $V_h(\cell)=\{v_{|\cell}\,:\,v \in V_h\,,\; \cell \in \mesh\}$ and the orthogonality is understood in $L^2(K;\R^{d \times d})$ with the inner product
	induced by ``$:$''.
\end{itemize}
Practical constructions of $(V_h,Q_h,I_h,\stab_h)$ satisfying the above estimates are described in \cite{HDM_linear}.

\begin{definition}[HD using GR]\label{def:HDM:GR}
	The HD based on a quadruplet $(V_h,Q_h,I_h,\stab_h)$ satisfying \propty{0}--\propty{5}
	is defined by: $X_{\disc,0}=V_h$ and, for $u\in X_{\disc,0}$,
	\[
	\Pi_\disc u=u\,,\;\nabla_\disc u = Q_h\nabla u\mbox{ and }\hd u=\nabla (Q_h \nabla u)+\stab_h\otimes (Q_h\nabla u-\nabla u).
	\]
\end{definition}
\section{Properties of HDM}\label{sec.properties}
This section describes the four properties associated with an HD and verifies these properties for several numerical methods so that the method fits in the HDM framework. The convergence analysis of a HS is based on four quantities and associated notions, measuring the stability and accuracy of the chosen HD.

The first quantity is a constant, $C_\disc$, that ensures discrete Poincar\'e inequalities. It is defined by 
\be\label{def.CD}
C_\disc = \max_{w\in X_{\disc,0}\setminus\{0\}} \left(\frac{\norm{\Pi_\disc w}{}}{\norm{\hd w}{}},
\frac{\norm{\nabla_\disc w}{L^4}}{\norm{\hd w}{}}\right).
\ee

The second quantity is the interpolation error $S_\disc$ defined by: for all $\varphi\in H^2_0(\O)$,
\be\label{def.SD}
\begin{aligned}
	&S_\disc(\varphi)=\min_{w\in X_{\disc,0}}\Big(\norm{\Pi_\disc w-\varphi}{}
	+\norm{\nabla_\disc w-\nabla\varphi}{L^4}	+\norm{\hd w-\hessian \varphi}{}\Big).
\end{aligned}
\ee
To define the limit-conformity measure for the Hessian scheme, introduce  $\wdspace(\O)=\{\xi\in L^2(\O;\R^{d \times d})\,;\,\hessian:\xi \in L^2(\O) \}$ and $ H_{\rm{div}}({\O})=\{\phi \in L^2(\O;\R^d):\, \mbox{div}\phi \in L^2(\O)\}$. For all $\xi \in \wdspace(\O)$ and $\phi \in H_{\rm{div}}({\O})$, set
\be\label{def.WD}
\begin{aligned}
	&W_\disc(\xi)=\max_{w\in X_{\disc,0}\backslash\{0\}}
	\frac{1}{\norm{\hd w}{}}\Big|\int_\O \Big((\hessian:\xi)\Pi_\disc w - \xi:\hd w \Big)\d\x \Big|, 
\end{aligned}
\ee
\be\label{def.WD_gradient}
\begin{aligned}
	& \hat{W}_\disc(\phi)=\max_{w\in X_{\disc,0}\backslash\{0\}}
	\frac{1}{\norm{\hd w}{}}\Big|\int_\O \Big(\nabla_\disc w \cdot \phi + \Pi_\disc w \mbox{ div} \phi \Big)\d\x \Big|. \end{aligned}
\ee
Here $W_\disc$ measures the defect of a double integration by parts \cite{HDM_linear} and is  the limit-conformity measure between the reconstructed Hessian and reconstructed function. $\hat{W}_\disc$ measures the defect of a Stokes formula between the reconstructed gradient and function. 

\begin{definition}[Coercivity, consistency, limit-conformity and compactness]\label{def:coercive}
	Let $(\disc_m)_{m \in \N}$ be a sequence of HDs in the sense of Definition \ref{HD}. We say that
	\begin{enumerate}
		\item $(\disc_m)_{m\in\N}$ is \emph{coercive} if there exists $C_P \in \R^+$ such that $C_{\disc_{m}} \leq C_P$ for all $m \in \N$.
		\item $(\disc_m)_{m\in\N}$ is \emph{consistent}, if 
		\be\label{def:cons}
		\forall \varphi\in H^2_0(\O)\,,
		\lim_{m \rightarrow \infty} S_{\disc_m}(\varphi)=0.
		\ee
		\item $(\disc_m)_{m\in\N}$ is \emph{limit-conforming}, if 
		\be\label{def:lc}
			\forall \: \xi \in H(\O), \, \forall \phi \in  H_{\rm{div}}({\O}),\quad
			\lim_{m \rightarrow \infty} \big(W_{\disc_m}(\xi)+\hat{W}_{\disc_m}(\phi)\big)=0.
		\ee
		\item $(\disc_m)_{m \in \N}$ is \emph{compact} if for any sequence $(u_m)_{m \in \N}$ such that $u_m \in X_{\disc_m,0}$ and $(\norm{\hessian_{\disc_m} u_m}{})_{m \in \N}$ is bounded, the sequence $(\Pi_{\disc_m}u_m)_{m \in \N}$ is relatively compact in $L^2(\O)$, and the sequence $(\nabla_{\disc_m}u_m)_{m \in \N}$ is relatively compact in $L^4(\O;\R^d)$.
	\end{enumerate}
\end{definition}

\begin{remark}\label{remark.dense}
	 As for the (second order) gradient discretisation method, see \cite[Lemmas 2.16 and 2.17]{gdm}, it can be easily proved that, for coercive sequences of HDs, the consistency and
	 limit-conformity properties only need to be tested for functions in dense subsets of $H^2_0(\O)$, and $H(\O)$ and $H_\div(\O)$, respectively.
\end{remark}
\begin{remark}[Comparison with the linear setting]\label{remark.linear}
	For linear equations, $C_\disc$ and $S_\disc$ are defined using the $L^2$-norms of the gradients. Dealing with the trilinear non-linearity requires higher integrability properties, and thus the usage the $L^4$-norms of gradients in the definitions \eqref{def.CD} and \eqref{def.SD} of $C_\disc$ and $S_\disc$, respectively. 
	
	Another difference with the linear setting is the introduction of $\hat{W}_\disc$ here. The limit-conformity defect $W_{\disc}$ is sufficient to analyse the convergence of the HDM for linear models. Here, however, the non-linear model \eqref{abstract_weak} involves the gradient, and accounting for $\hat{W}_\disc$ in the definition of limit-conformity is necessary to identify the limit of the reconstructed gradients in the convergence analysis.
	
\end{remark}
\begin{remark}
	In most cases, by the continuous Sobolev embedding (which is often also valid at the discrete level \cite[Appendix B]{gdm}) we actually expect $(\Pi_{\disc_m}u_m)_{m \in \N}$ and $(\nabla_{\disc_m}u_m)_{m \in \N}$ to be compact in $L^p$ for all $p < 2^*$, where $2^*$ is a Sobolev exponent associated with $2$.
\end{remark}
\subsection{Properties of numerical methods in HDM framework}
A few methods that fits in the HDM framework is considered in Section \ref{sec:class.in.HDM}. Here we will state the properties, namely coercivity, consistency, limit-conformity and compactness, in Definition \ref{def:coercive} in the context of these numerical methods. The proofs associated with the ncFEMs and GR methods are provided in the Appendix.
\subsubsection{\bf Conforming FEM}
The estimates on $C_\disc$,
$S_\disc$, $W_\disc$, $\hat{W}_\disc$ and the compactness property easily follow:
\begin{itemize}
	\item $C_\disc$ is bounded by the maximum of the constants of the continuous Poincar\'e inequality in $H^2_0(\O)$ and the continuous Sobolev imbedding $H^1(\O)\hookrightarrow L^4(\O)$.
	\item  Standard interpolation properties (see, e.g., \cite{ciarlet1978finite}) and the continuous Sobolev imbedding $H^1(\O)\hookrightarrow L^4(\O)$ yield an $\mathcal O(h)$ estimate on $S_\disc(\varphi)$, provided $\varphi\in H^3(\Omega)\cap H^2_0(\Omega)$. This and Remark \ref{remark.dense} imply that $\lim_{h\to 0}S_\disc(\varphi)\to 0$ for all $\varphi\in H^2_0(\Omega)$. 
	\item Integration-by-parts in $H^2_0(\O)$ shows that $W_\disc(\xi)=0$ for all $\xi \in \wdspace(\O)$ and $\hat{W}_\disc(\phi)=0$ for all $\phi \in H_{\rm{div}}({\O})$.
	\item The compactness of $(\disc_m)_{m \in \N}$ follows from the Rellich and Sobolev imbedding theorems.
\end{itemize}

\subsubsection{\bf Non-conforming FEM}\label{sec.ncfem.properties}
Theorem \ref{th.MorleyAdini} provides estimates on the four quantities associated with HD and shows that, along sequences of refined meshes, the HDs corresponding to the Morley and Adini finite elements satisfy the coercivity, consistency, limit-conformity and compactness properties. The proof is provided in the Appendix. These properties are essential to apply Theorem \ref{convergence_abstract_non-linear}. 

\begin{theorem}\label{th.MorleyAdini}
	Let $\disc$ be a HD for the Morley (resp.\ Adini) ncFEM in the sense of Definition \ref{def.HD.Morley} (resp. Definition \ref{def.HD.Adini}). Then the following hold:
	\begin{itemize}
		\item[(i)] (Coercivity) $C_\disc {\lesssim 1}$,
		\item[(ii)] (Consistency) $ \forall\: \varphi \in H^{3}(\O)\cap H^2_0(\O)$, $S_{\disc}(\varphi) \lesssim h\norm{\varphi}{H^{3}(\O)}$,
		\item[(iii)] (Limit-conformity) $\forall \:\xi \in H^2(\O;\R^{2 \times 2}), \forall \phi \in H^1(\O;\R^2),$ 
		$$
		W_{\disc}(\xi)+\hat{W}_{\disc}(\phi)\lesssim h\big( \norm{\xi}{H^2(\O;\R^{2 \times 2})} + \norm{\phi}{H^1(\O)}\big),
		$$
		\item[(iv)] (Compactness) For a sequence of meshes $(\mesh_{h_m})_{m \in \N}$ with $h_m \rightarrow 0$, denoting the HD constructed on $\mesh_{h_m}$ by $\disc_m$, the sequence $(\disc_m)_{m \in \N}$ is compact.
	\end{itemize}
\end{theorem}	
As a consequence, following Remark \ref{remark.dense}, if $(\mesh_{m})_{m \in \N}$ is a regular family of meshes and $\disc_m$ is the HD for the Adini and Morley ncFEMs on $\mesh_{m}$, then $(\disc_m)_{m \in \N}$ is coercive, consistent, limit-conforming and compact.
\subsubsection{\bf Method based on GR Operators}\label{sec.grmethod.properties}
Recall the properties $\propty{0}-\propty{5}$ method based on GR operators associated with Section \ref{sec.grmethod}. 

\begin{theorem}[Estimates for HDs based on GR]\label{thm:HD.GR}
	Let $\disc$ be a HD in the sense of Definition \ref{def:HDM:GR} such that $(V_h,I_h,Q_h,\stab_h)$ satisfy \propty{0}--\propty{5}.
	Then,
	\begin{itemize}
		\item[(i)]  (Coercivity) $C_\disc \lesssim 1$, 
		\item[(ii)](Consistency) $ \forall\: \varphi \in W$, $S_{\disc}(\varphi) \lesssim  h\norm{\varphi}{W}$,
		\item[(iii)] (Limit-conformity) $\forall \:\xi \in H^2(\O;\R^{d \times d})$, $W_{\disc}(\xi) \lesssim h \norm{\xi}{H^2(\O;\R^{d \times d})}$ and $\forall \phi \in H_{\rm{div}}(\O)$, $\hat{W}_{\disc}(\phi)=0,$
		\item[(iv)] (Compactness) If $(\mesh_m)_{m\in\N}$ is a sequence of meshes and $\disc_m$ is a GR HD based on $\mesh_m$ for discrete elements satisfying \propty{0}--\propty{5} uniformly with respect to $m$, then $(\disc_m)_{m\in\N}$ is compact.
	\end{itemize}
\end{theorem}
As a consequence, following Remark \ref{remark.dense}, if $(\mesh_{m})_{m \in \N}$ is a regular family of meshes and $\disc_m$ is the HD for the GR method on $\mesh_{m}$, then $(\disc_m)_{m \in \N}$ is coercive, consistent, limit-conforming and compact.
\section{Convergence analysis} \label{con.reg}
We establish the convergence of the HS, provided the underlying sequences of HDs satisfy the properties in Definition \ref{def:coercive}. This convergence is proved without any extra regularity assumption on the exact solution, or the assumption that the linearized problem around this solution is well-posed.
Let us start with a preliminary lemma.

\begin{lemma}[Regularity of the limit] \label{regularity_nonlinear}
	Let $(\disc_m)_{m \in \N}$ be a  coercive and limit-conforming sequen\-ce of HDs in the sense of Definition \ref{def:coercive}$(i)$ and $(iii)$. Let $u_m\in X_{\disc_m,0}$ be such that $\norm{u_m}{\disc_m}$ remains bounded. Then, there exists a subsequence of $(\disc_m,u_m)_{m \in \N}$ (denoted using the same notation)  and $u \in H^2_0(\O)$ such that $\Pi_{\disc_m}u_m$ converges weakly to $u$ in $L^2(\O)$, $\nabla_{\disc_m}u_m$ converges weakly to $\nabla u$ in $L^4(\O;\R^d)$, and $\hessian_{\disc_m}u_m$ converges weakly to $\hessian u$ in $L^2(\O; \R^{d \times d})$.
\end{lemma}

\begin{proof}
The bound on $\norm{u_m}{\disc_m}$ from coercivity of $(\disc_m)_{m \in \N}$ implies that $(\Pi_{\disc_m}u_m)_m$ and $(\nabla_{\disc_m}u_m)_m$ are bounded in $L^2(\O)$ and $L^4(\O;\R^d)$, respectively. Therefore, there exists a subsequence of $(\disc_m,u_m)_{m \in \N}$ and $u \in L^2(\O), v \in L^4(\O;\R^d)$ and $w \in L^2(\O; \R^{d \times d})$ such that $\Pi_{\disc_m}u_m$ converges weakly in $L^2(\O)$ to $u$, $\nabla_{\disc_m}u_m$ converges weakly in $L^4(\O;\R^d)$ to $v$, and $\hessian_{\disc_m}u_m$ converges weakly in $L^2(\O;\R^{d \times d})$ to $w$. It remains to prove that $v=\nabla u$, $w=\hessian u$ and $u \in H^2_0(\O)$. We extend $\Pi_{\disc_m}u_m, u, \nabla_{\disc_m}u_m, v, \hessian_{\disc_m}u_m$ and $w$ by 0 outside $\O$, and the same convergence results hold, respectively, in $L^2(\R^d)$, $L^4(\R^d)^d$ and $L^2(\R^d;\R^{d \times d})$. Using the limit-conformity of $(\disc_m)_{m \in \N}$ and the bound on $\norm{u_m}{\disc_m}$, passing to the limit in \eqref{def.WD}-\eqref{def.WD_gradient} gives
	\be \label{Hess_int}
	\forall \xi \in H({\R^d}), \, \int_{\R^d}^{}\big((\hessian:\xi)u-\xi:w\big)\d\x=0
	\ee
	\be \label{Grad_int}
	\mbox{and   }	\forall  \phi \in H_{\rm{div}}({\R^d}), \, \int_{\R^d}^{}\big(v \cdot \phi + u\,\mbox{div}\phi\big)\d\x=0.
	\ee
	For $\phi \in C_c^\infty({\R^d})^{d}$ and $\xi \in C_c^\infty({\R^d};\R^{d \times d})$, \eqref{Hess_int} and \eqref{Grad_int} show that $w=\hessian u$ and $v=\nabla u$, in the sense of distributions on $\R^d$. This implies $u \in H^2(\R^d)$ and, since $u=0$ outside the domain $\O$, $u \in H^2_0(\O)$.
\end{proof}
\begin{theorem}[Convergence of the HDM] \label{convergence_abstract_non-linear}
	Let $(\disc_m)_{m \in \N}$ be a sequence of HDs, in the sense of Definition \ref{HD}, that is coercive, consistent, limit-conforming and compact in the sense of Definition \ref{def:coercive}. Then, for any $m \in \N$, there exists at least one weak solution $\Psi_{\dm}$ of \eqref{abstract_HS}, with $\disc=\disc_m$. Moreover, as $m \rightarrow \infty$, there exist a subsequence of $(\disc_m)_{m \in \N}$ (denoted using the same notation $(\disc_m)_{m \in \N}$), and a solution $\Psi$ of the abstract problem \eqref{abstract_weak} such that $\Pi_{\disc_m}\Psi_{\dm} \rightarrow\Psi$ in $\bsymb{L}^2(\O)$, $\nabla_{\disc_m}\Psi_{\dm}\rightarrow\nabla \Psi$ in $\bsymb{L}^4(\O;\R^d)$ and $\hessian_{\disc_m}\Psi_{\dm}\rightarrow\hessian \Psi$ in $\bsymb{L}^2(\O;\symd)$.
\end{theorem}
\begin{proof}[Proof of Theorem \ref{convergence_abstract_non-linear}] The proof is divided into four steps. 
	
	\textbf{Step 1: }existence of a solution to the scheme.\\
	For any HD $\disc$, let $\overline{\Psi}_\disc \in \bfXd$ be given and ${\Psi_\disc} \in \bfXd$ be such that, for all $\Phi_\disc \in \bfXd$,
	\begin{align}
	&\mathcal{A}_{\overline{\Psi}_\disc}(\Psi_\disc,\Phi_\disc):=\mathcal{A}(\hd \Psi_\disc, \hd \Phi_\disc)+\mathcal{B}(\hd \overline{\Psi}_\disc, \nabla_\disc\Psi_\disc,\nabla_\disc\Phi_\disc)=\mathcal{L}(\Pi_\disc\Phi_\disc). \label{abstract_HS2}
	\end{align}
	Since $\mathcal{A}(\cdot,\cdot)$ is bilinear, $\mathcal{B}(\cdot,\cdot,\cdot)$ is trilinear and $\overline{\Psi}_\disc \in \bfXd$ is fixed,  $\mathcal{A}_{\overline{\Psi}_\disc}(\cdot,\cdot)$ is bilinear.
	Therefore, ${\Psi_\disc}$ is sought as a solution to the bilinear system $\mathcal{A}_{\overline{\Psi}_\disc}(\Psi_\disc,\Phi_\disc)= \mathcal{L}(\Pi_\disc\Phi_\disc)$. Since  $\bfXd$ is finite-dimensional  and $ \mathcal{L}(\Pi_\disc\cdot)$ is linear,  $ \mathcal{L}(\Pi_\disc\cdot)$ is a continuous linear functional on $\bfXd$. Use the fact that $\mathcal{B}(\hd \overline{\Psi}_\disc, \nabla_\disc\Psi_\disc,\nabla_\disc\Psi_\disc)=0$ (see \assum{3}) and $\mathcal{A}(\cdot,\cdot)$ is coercive, to infer that
	\begin{align}
	\mathcal{A}_{\overline{\Psi}_\disc}(\Psi_\disc,\Psi_\disc)=\mathcal{A}(\hd \Psi_\disc,\hd \Psi_\disc)\ge\overline{\alpha}\norm{\hd \Psi_\disc}{}^2=\overline{\alpha}\norm{\Psi_\disc}{\disc}^2, \label{abstract_HS_coercive}
	\end{align}
	where $\overline{\alpha}$ is the coercivity constant of $\mathcal{A}(\cdot,\cdot)$. Thus, $\mathcal{A}_{\overline{\Psi}_\disc}(\cdot,\cdot)$ is coercive. The Lax Milgram Lemma implies the existence and uniqueness of solution $\Psi_\disc$ satisfying \eqref{abstract_HS2}.
	Define $F: \bfXd\rightarrow \bfXd$ by $F(\overline{\Psi}_\disc)=\Psi_\disc$, where $\Psi_\disc$ is the solution to \eqref{abstract_HS2}. Since  $\bfXd$ is finite dimensional, we can easily check that $F$ is continuous. Moreover,  \eqref{abstract_HS_coercive} and \eqref{abstract_HS2} imply,
	$$
	\overline{\alpha}\norm{\Psi_\disc}{\disc}^2 \le \mathcal{A}_{\overline{\Psi}_\disc}(\Psi_\disc,\Psi_\disc) = \mathcal{L}(\Pi_\disc\Psi_\disc)\le \norm{\mathcal{L}}{}\norm{\Pi_\disc\Psi_\disc}{} \le C_\disc \norm{\mathcal{L}}{}\norm{\Psi_\disc}{\disc},$$
	where $C_\disc$ is defined by \eqref{def.CD}. Hence,
	\be 
	\norm{\Psi_\disc}{\disc} \le\overline{\alpha}^{-1}C_\disc \norm{\mathcal{L}}{}:=\mathnormal{R}_\disc.\label{abstract_HS_stability}
	\ee
	This shows that $F$ maps $ \bfXd$ into the closed ball $B_{\mathnormal{R}_\disc}$ of center 0 and radius $\mathnormal{R}_\disc$ with respect to  $\norm{{\cdot}}{\disc}$. Therefore, the Brouwer fixed point theorem proves that $F$ has at least one fixed point $\Psi_\disc$ in this ball. The equation \eqref{abstract_HS2} shows that this fixed point is a solution to \eqref{abstract_HS}. 
	
	From here onwards, let $\Psi_{\dm} \in  \bsymb{X}_{\disc_m,0}$ denote such a solution for $\disc=\disc_m.$
	
	\medskip
	
	\textbf{Step 2:} strong convergence of $\Pi_{\disc_m}\Psi_{\dm}$ and $\nabla_{\disc_m}\Psi_{\dm}$, and weak convergence of $\hessian_{\disc_m}\Psi_{\dm}$.
	
	From \eqref{abstract_HS_stability}, $\overline{\alpha}\norm{\hessian_{\dm} \Psi_{\dm}}{}=\overline{\alpha} \norm{\Psi_{\dm}}{\dm}\le  C_{\dm} \norm{\mathcal{L}}{}.$
	Thus,  $\norm{\Psi_{\dm}}{\disc_m}$ is bounded and Lemma \ref{regularity_nonlinear} gives a subsequence of $(\disc_m,\Psi_{\dm})_{m \in \N}$, and $\Psi \in \bfX$, such that $\Pi_{\disc_m} \Psi_{\dm}$ converges weakly to $\Psi$ in $\bsymb{L}^2(\O)$, $\nabla_{\disc_m} \Psi_{\dm}$ converges weakly to $\nabla \Psi$ in $\bsymb{L}^4(\O;\R^d)$, and $\hessian_{\disc_m} \Psi_{\dm}$ converges weakly to $\hessian \Psi$ in $\bsymb{L}^2(\O;\symd)$. The compactness hypothesis (Defintion \ref{def:coercive}) then shows the strong convergence of $\Pi_{\disc_m} \Psi_{\dm}$ to $\Psi$ in $\bsymb{L}^2(\O)$ and $\nabla_{\disc_m} \Psi_{\dm}$ to $\nabla \Psi$ in $\bsymb{L}^4(\O;\R^d).$
	
	\medskip
	
	\textbf{Step 3:} $\Psi$ is a solution to Problem \eqref{abstract_weak}.
	
Define $P_\disc:\bfX \rightarrow \bfXd$ by
			\begin{equation}\label{def:PD}
				P_\disc \Psi= \argmin_{w \in \bfXd}\big(\norm{\Pi_\disc w-\Psi}{}
			+\norm{\nabla_\disc w-\nabla\Psi}{L^4}	+\norm{\hd w-\hessian \Psi}{}\big)
		\end{equation}
	and let $\Phi\in \bfX$. The consistency of $(\disc_m)_{m \in \N}$ implies  $\Pi_{\disc_m}P_{\dm}\Phi \rightarrow \Phi$ in $\bsymb{L}^2(\O)$, $\nabla_{\disc_m}P_{\dm}\Phi \rightarrow \nabla\Phi$ in $\bsymb{L}^4(\O;\R^d)$ and $\hessian_{\disc_m}P_{\dm}\Phi \rightarrow \hessian\Phi$ in $\bsymb{L}^2(\O;\symd)$ as $m \rightarrow \infty$.
	 \begin{align*}
		&\mathcal{B}(\hessian_{\dm}{\Psi}_{\dm},\nabla_{\dm}\Psi_{\dm},\nabla_{\dm}P_{\dm}\Phi)-\mathcal{B}(\hessian \Psi, \nabla \Psi, \nabla \Phi)\nonumber\\
		&=\mathcal{B}(\hessian_{\dm}{\Psi}_{\dm},\nabla_{\dm}\Psi_{\dm},\nabla_{\dm}P_{\dm}\Phi-\nabla \Phi)\nonumber\\
		&\quad+\mathcal{B}(\hessian_{\dm}{\Psi}_{\dm},\nabla_{\dm}\Psi_{\dm}-\nabla \Psi,\nabla \Phi)+\mathcal{B}(\hessian_{\dm}{\Psi}_{\dm}-\hessian \Psi,\nabla\Psi,\nabla \Phi).	
		\end{align*}
		Set $l(\hessian_{\dm}{\Psi}_{\dm})=\mathcal{B}(\hessian_{\dm}{\Psi}_{\dm}, \nabla \Psi, \nabla \Phi).$ Since $\mathcal{B}(\cdot,\cdot,\cdot)$ is a trilinear continuous function, ${l}(\cdot)$ is a linear continuous functional on $\bsymb{L}^2(\O;\symd)$. The weak convergence of $(\hessian_{\dm}{\Psi}_{\dm})_{m \in \N}$ then ensures that ${l}(\hessian_{\dm}{\Psi}_{\dm}) \rightarrow {l}(\hessian \Psi)$ as $m \rightarrow \infty$. The continuity of $\mathcal{B}(\cdot,\cdot,\cdot)$ yields a constant $C_b$ such that
		\begin{align*}
		&\big|\mathcal{B}(\hessian_{\dm}{\Psi}_{\dm},\nabla_{\dm}\Psi_{\dm},\nabla_{\dm}P_{\dm}\Phi)-\mathcal{B}(\hessian \Psi, \nabla \Psi, \nabla \Phi)\big|\\
		& \le C_b\norm{\hessian_{\dm}{\Psi}_{\dm}}{}\norm{\nabla_{\dm}\Psi_{\dm}}{{L}^4}\norm{\nabla_{\dm}P_{\dm}\Phi-\nabla \Phi}{{L}^4}
		\\
		&\quad+C_b\norm{\hessian_{\dm}{\Psi}_{\dm}}{}\norm{\nabla_{\dm}\Psi_{\dm}-\nabla \Psi}{{L}^4}\norm{\nabla \Phi}{{L}^4}+\big|{l}(\hessian_{\dm}{\Psi}_{\dm})-{l}(\hessian \Psi)\big|.
		\end{align*}
		Since strongly/weakly convergent sequences in normed space are bounded, the convergences of $(\nabla_{\dm}P_{\dm}\Phi)_{m\in\N}$, $(\nabla_{\dm}\Psi_{\dm})_{m\in\N}$ and $(l(\hessian_{\dm}{\Psi}_{\dm}))_{m\in\N}$ imply that 
		\be 
		\mathcal{B}(\hessian_{\dm}{\Psi}_{\dm},\nabla_{\dm}\Psi_{\dm},\nabla_{\dm}P_{\dm}\Phi) \rightarrow \mathcal{B}(\hessian \Psi, \nabla \Psi, \nabla \Phi) \mbox{ as } m \rightarrow \infty. \nonumber
		\ee
		 This, the bilinearity and continuity of $\mathcal A$ show that, as $m \rightarrow \infty$,
	\begin{align}
	\mathcal{A}(\hessian_{\dm}\Psi_{\dm},\hessian_{\dm}P_{\dm}\Phi)+\mathcal{B}(\hessian_{\dm}&{\Psi}_{\dm},\nabla_{\dm}\Psi_{\dm},\nabla_{\dm}P_{\dm}\Phi)\nonumber\\
	& \rightarrow \mathcal{A}(\hessian \Psi,\hessian \Phi)+\mathcal{B}(\hessian \Psi, \nabla \Psi, \nabla \Phi).\label{abstract_passlimit1}
	\end{align}
	
	Moreover, a direct consequence of $\Pi_{\disc_m}P_{\dm}\Phi \rightarrow \Phi$ in $\bsymb{L}^2(\O)$ as $m \rightarrow \infty$ shows that
	\be
	\mathcal{L}(\Pi_{\disc_m}P_{\dm}\Phi) \rightarrow \mathcal{L}( \Phi) \mbox{  as }m \rightarrow \infty. \label{abstract_passlimit2}
	\ee
	
	Letting $\Phi_{\dm}=P_{\dm}\Phi$ in \eqref{abstract_HS} for $\disc=\dm$,  use \eqref{abstract_passlimit1} and \eqref{abstract_passlimit2} to pass to the limit and conclude that $\Psi$ is a solution to \eqref{abstract_weak}.
	
	\medskip
	
	\textbf{Step 4:} strong convergence of $\hessian_{\disc_m}\Psi_{\dm}$.
	
	The strong convergence of $\Pi_\disc\Psi_{\dm}$ and \assum{3} enable us to pass to the limit in \eqref{abstract_HS} for $\disc=\disc_m$ to see that
	\begin{align*}
	\lim\limits_{m \rightarrow \infty}\mathcal{A}(\hd \Psi_{\dm},\hd \Psi_{\dm})	= \lim\limits_{m \rightarrow\infty}\mathcal{L}(\Pi_\disc\Psi_{\dm}) =\mathcal{L}(\Psi)
	=\mathcal{A}(\hessian \Psi,\hessian \Psi),
	\end{align*}
	since $\Psi$ is a solution to \eqref{abstract_weak}. The coercivity and bilinearity of $\mathcal A$, and the weak convergence of $\hessian_{\dm}\Psi_{\dm}$ therefore lead to
	\begin{align*}
	\limsup_{m \rightarrow \infty}\overline{\alpha}	\norm{\hessian_{\disc_m}\Psi_{\dm}-\hessian \Psi}{}^2\le	\limsup\limits_{m \rightarrow \infty}\mathcal{A} (\hessian_{\disc_m}\Psi_{\dm}-\hessian \Psi, \hessian_{\disc_m}\Psi_{\dm}-\hessian \Psi) &=0.
	\end{align*}
	This shows that $\norm{\hessian_{\disc_m}\Psi_{\dm}-\hessian \Psi}{} \rightarrow 0 $  as $ m \rightarrow \infty$.
\end{proof}

\begin{remark}

As seen in Section \ref{sec:conf}, it is easy to construct a coercive, consistent, limit-conforming and compact sequence of HDs. A consequence of this and Theorem \ref{convergence_abstract_non-linear} leads to the existence of a solution to the abstract problem \eqref{abstract_weak} that in particular applies to the the stream function vorticity formulation of the incompressible 2D Navier--Stokes problem and von K\'{a}rm\'{a}n equations approximated using the conforming and nonconforming FEMs and the GR methods.
\end{remark}

\section{Numerical results}\label{sec.numericalresults}
This section deals with the numerical results for the Navier--Stokes (NS) equation in stream function vorticity formulation and the von K\'{a}rm\'{a}n (vK) equations using the GR method and the Morley ncFEM. Define 
\begin{align*}
&\err_\disc(\bu):=\frac{\norm{\Pi_\disc u_\disc -\bu}{}}{\norm{\bu}{}},\quad \err_\disc(\nabla\bu) :=\frac{\norm{\nabla_\disc u_\disc -\nabla\bu}{}}{\norm{\nabla\bu}{}}, \\
& \err_\disc(\hessian\bu) :=\frac{\norm{\hd u_\disc -\hessian\bu}{}}{\norm{\hessian\bu}{}},
\end{align*}
where $\bu$ is the continuous solution and $u_\disc$ is the corresponding HS solution. In the tables, $h$ and $\textbf{nu}$ denote the mesh size and the numbers of unknowns. The model problem is constructed in such a way that the exact solution is known. The discrete problem is solved using Newton's method. The uniform mesh refinement has been done by red-refinement criteria, where each triangle is subdivided into four sub-triangles by connecting the midpoints of the edges.

\subsection{Numerical results for GR Method}\label{numericalresultGR} Let the computational domain be $\Omega=(0,1)^2$, and consider a family of meshes made of uniform triangulations. The finite dimensional space $V_h$ is the conforming $\Poly_1$ space. It seems that the GR method was previously never considered for non-linear problem. 

%
\subsubsection{\bf Navier--Stokes equation}  
Let the exact solution be given by $\bu=x^2y^2(1-x)^2(1-y)^2$. Then choosing $\nu=1$, the load function is computed using
$$\Delta^2\bu + \frac{\partial}{\partial x}\bigg((-\Delta \bu)\frac{\partial \bu}{\partial y}\bigg)- \frac{\partial}{\partial y}\bigg((-\Delta \bu)\frac{\partial \bu}{\partial x}\bigg)=f.$$
The errors and orders of convergence for the numerical approximation of $\bu$ are presented in Table \ref{table1.grns}. As seen in the table, the rate of convergence is nearly quadratic in $L^2$ and $H^1$ norms and is linear in $H^2$ norm. These observed numerical rates of convergence are as expected for the GR method (in comparison with the linear case \cite{HDM_linear}). A rigorous proof of the theoretical orders of convergence is a proposed future work. 

\begin{table}[h!!]
	\caption{\small{(GR/NS) Convergence results for the relative errors of $\bu$}}
	{\small{\footnotesize
			\begin{center}
				\begin{tabular}{|c |c
						||c|c||c | c ||c|c|| c|c||c|c||c|c||}
					\hline
					$h$&	$\textbf{nu}$ &$\err_\disc(\bu)$ & Order  & $\err_\disc(\nabla \bu)$ & Order  &$\err_\disc( \hessian \bu)$ & Order  \\ 
					\hline\\[-10pt]  &&\\[-10pt]
					0.353553&9&    1.050933 &       - &    0.567673 &      - &    0.582651&        - \\
					0.176777 & 49&  0.214195  &  2.2947 &   0.167145  &  1.7640&    0.267188&    1.1248\\
					0.088388& 225&  0.067498  &  1.6660 &   0.049952  &  1.7425&    0.128511 &   1.0560\\
					0.044194& 961&   0.019240&    1.8107 &   0.013806&    1.8552 &   0.062184&   1.0473	\\
					0.022097& 3969&0.005156&1.8999&0.003646&1.9209& 0.030460&1.0296\\	
					0.011049&16129&0.001336&1.9482& 0.000939& 1.9575&0.015060&1.0162					
					
					\\	\hline				
				\end{tabular}
			\end{center}	}}\label{table1.grns}
		\end{table}										
		
		\subsubsection{\bf The von K\'{a}rm\'{a}n equations}
In this example, $\Omega=(0,1)^2$ and $\bu=\bv=x^2y^2(1-x)^2(1-y)^2$. Then the right hand side load functions are computed as $f=\Delta^2\bu-[\bu,\bv]$ and $g=\Delta^2\bv+\frac{1}{2}[\bu,\bu]$. Table \ref{gr.vk_v} shows the relative errors and orders of convergence for the variables $\bu$ and $\bv$. The table provides rates of convergence close to quadratic in $L^2$ and $H^1$ norms, and linear in the energy norm for both the variables. 
		\begin{table}[h!!]
			\caption{\small{(GR/vK) Convergence results for the relative errors of $\bu$ and $\bv$}}
			{\small{\footnotesize
					\begin{center}
						\begin{tabular}{||c |c
								||c|c||c | c ||c|c|| c|c||c|c||c|c||}
							\hline
							$h$&	$\textbf{nu}$&$\err_\disc(\bu)$ & Order  & $\err_\disc(\nabla \bu)$ & Order  &$\err_\disc( \hessian \bu)$ & Order  \\ 
							\hline\\[-10pt]  &&\\[-10pt]
							0.353553&9& 1.049207&   -&   0.567835&    -& 0.582623&    -\\
							0.176777&49&0.214594&  2.2896& 0.167636& 1.7601& 0.267284& 1.1242\\
							0.088388&225&0.067946& 1.6591&  0.050446& 1.7325&  0.128565&  1.0559\\
							0.044194&961&0.019702&1.7860& 0.014295& 1.8192&   0.062217& 1.0471\\
							0.022097&3969&0.005632&1.8068&0.004146&1.7858&0.030483&1.0293\\
							0.011049&16129&0.001844&1.6109&0.001483&1.4828&0.015082& 1.0152\\
							\hline
						\end{tabular}
						\begin{tabular}{ ||c|c||c |c||c| c|| c|c||c|c||c|c|}
							\hline
							$h$&$\textbf{nu}$&	$\err_\disc(\bv)$ & Order&$\err_\disc(\nabla \bv)$ & Order  &$\err_\disc(\hessian\bv)$ & Order  \\ 
							\hline\\[-10pt]  &&\\[-10pt]
							0.353553&9&1.051793&    -&0.567587&     -& 0.582660&    -\\
							0.176777&49&0.213996&  2.2972&0.166900& 1.7659& 0.267141& 1.1251\\
							0.088388&225& 0.067275& 1.6694&0.049707&  1.7475&   0.128485& 1.0560\\
							0.044194&961&0.019011& 1.8232& 0.013567&   1.8733& 0.062171& 1.0473\\
							0.022097&3969&0.004929&1.9474& 0.003417&1.9894&0.030454&1.0296\\
							0.011049&16129&0.001124&2.1325&0.000742&2.2036& 0.015059& 1.0160\\
							\hline							
						\end{tabular}
					\end{center}}	}\label{gr.vk_v}
				\end{table}

\subsection{Numerical results for Morley FEM}

\subsubsection{\bf Navier--Stokes equation}\label{eg.ns} For the example considered in Section 7.1.1, the errors and order of convergences are presented in Table \ref{table1.morleyns}. Observe that a linear order of convergence is obtained for $\bu$ in the energy norm, and quadratic orders of convergence are obtained in piecewise $H^1$ and $L^2$ norms. 
\begin{table}[h!!]
	\caption{\small{(Morley/NS) Convergence results for the relative errors of $\bu$}}
	{\small{\footnotesize
			\begin{center}
				\begin{tabular}{|c |c
						||c|c||c | c ||c|c|| c|c||c|c||c|c||}
					\hline
					$h$&	$\textbf{nu}$ &$\err_\disc(\bu)$ & Order  & $\err_\disc(\nabla \bu)$ & Order  &$\err_\disc( \hessian \bu)$ & Order  \\ 
					\hline\\[-10pt]  &&\\[-10pt]
					1.00000& 5& 0.0135922&    -& 0.027680&     -& 0.147997&          -\\
					0.50000&25& 0.003499& 1.9579& 0.008910& 1.6353& 0.083508& 0.8256\\
					0.25000&113& 0.000923& 1.9225&  0.002578& 1.7890& 0.042875& 0.9618\\
					0.12500&481& 0.000246& 1.9102&  0.000720& 1.8406& 0.022240& 0.9470\\
					0.06250&1985&0.000063& 1.9700&0.000187& 1.9472& 0.011261&0.9818\\
					0.03125&8065&0.000016&  1.9918& 0.000047& 1.9855& 0.005650& 0.9950
					
					\\	\hline				
				\end{tabular}
			\end{center}	}}\label{table1.morleyns}
		\end{table}										
				\subsubsection{\bf The von K\'{a}rm\'{a}n equations}\label{sec.vkMorley}
				
				The results of numerical experiments for the Morley ncFEM for the von K\'{a}rm\'{a}n equations are presented here, as the formulation in this article is different from that in \cite{Brezzi,ng1,ng2} (see Remark \ref{remark.vk.wf}).
				\smallskip
				
				\textsc{Example 1.}
			In this example, choose the data as in Section 7.1.2.
				\begin{table}[h!!]
					\caption{\small{(Morley/vK) Convergence results for the relative errors of $\bu$ and $\bv$, Example 1}}
					{\small{\footnotesize
							\begin{center}
								\begin{tabular}{|c |c
										||c|c||c | c ||c|c|| c|c||c|c||c|c||}
									\hline
									$h$&	$\textbf{nu}$ &$\err_\disc(\bu)$ & Order  & $\err_\disc(\nabla \bu)$ & Order  &$\err_\disc( \hessian \bu)$ & Order  \\ 
									\hline\\[-10pt]  &&\\[-10pt]
1.00000&5&  8.560933&   -&3.564432&     -& 2.585671&   -\\
0.50000&25& 2.204201&1.9575&1.145871&1.6372& 1.461266&0.8233\\
0.25000&113&0.581424&1.9226&0.331537&1.7892& 0.750127&0.9620\\
0.12500&481 &0.154705&1.9101&0.092576&1.8405& 0.389103&0.9470\\
0.06250&1985 &0.039490&1.9700&0.024008&1.9471&0.197022&0.9818\\	0.03125& 8065&0.009929&1.9918&0.006062&1.9855& 0.098852&0.9950					
									\\	\hline				
								\end{tabular}
								\begin{tabular}{ |c|c||c |c||c| c|| c|c||c|c||c|c|}
									\hline
									$h$&	$\textbf{nu}$&$\err_\disc(\bv)$ & Order&$\err_\disc(\nabla \bv)$ & Order  &$\err_\disc(\hessian\bv)$ & Order  \\ 
									\hline\\[-10pt]  &&\\[-10pt]
1.00000&5&8.564924&-&  3.566189&-& 2.586783&-\\
0.50000&25&2.204151&1.9582&1.145773& 1.6381& 1.461500&0.8237\\
0.25000&113&0.581494&1.9224&0.331586& 1.7889& 0.750404&0.9617\\
0.12500&481 &0.154705&1.9102&0.092575& 1.8407& 0.389239&0.9470\\
0.06250&1985 &0.039488&1.9700&0.024007&1.9472& 0.197091&0.9818\\
0.03125& 8065&0.009928&1.9918&0.006062&1.9855& 0.098886&0.9950\\							
									\hline				
								\end{tabular}
							\end{center}}	}\label{table2}
						\end{table}
						As seen in the Table \ref{table2}, the order of convergence in the energy norm (resp. $H^1$ and $L^2$ norms) is linear (resp. quadratic) for the displacement and Airy stress functions.
						\smallskip
						
						\textsc{Example 2.} Consider the L-shaped domain $\Omega=(-1,1)^2 \setminus\big{(}[0,1)\times(-1,0]\big{)}$. 
						 Choose the right hand functions such that the exact singular solution \cite{Grisvard92} in polar coordinates is given by
						\begin{align*}
						\bu=\bv=(r^2 \cos^2\theta-1)^2 (r^2 \sin^2\theta-1)^2 r^{1+ \gamma}g_{\gamma,\omega}(\theta),
						\end{align*}
						where $ \gamma\approx 0.5444837367$ is a non-characteristic 
						root of $\sin^2( \gamma\omega) =  \gamma^2\sin^2(\omega)$, $\omega=\frac{3\pi}{2}$,  and
						$g_{\gamma,\omega}(\theta)=(\frac{1}{\gamma-1}\sin ((\gamma-1)\omega)-\frac{1}{ \gamma+1}\sin(( \gamma+1)\omega))(\cos(( \gamma-1)\theta)-\cos(( \gamma+1)\theta))$ 
						$-(\frac{1}{\gamma-1}\sin(( \gamma-1)\theta)-\frac{1}{ \gamma+1}\sin(( \gamma+1)\theta))
						(\cos(( \gamma-1)\omega)-\cos(( \gamma+1)\omega)).$ 
						\begin{table}[h!!]
							\caption{\small{(Morley/vK) Convergence results for the relative errors of $\bu$ and $\bv$, Example 2}}
							{\small{\footnotesize
									\begin{center}
										\begin{tabular}{ |c|c||c | c ||c|c|| c|c||c|c||c|c||}
											\hline
											$h$ &	$\textbf{nu}$ &$\err_\disc(\bu)$ & Order  & $\err_\disc(\nabla \bu)$ & Order  &$\err_\disc(\hessian \bu)$ & Order  \\ 
											\hline\\[-10pt]  &&\\[-10pt]
											0.707107& 33&  2.826994&-& 1.985957& -&  1.758240&-\\
											
											0.353553&161& 0.874885& 1.6921& 0.623930& 1.6704&  0.984743&  0.8363\\
											
											0.176777&705& 0.250204&1.8060& 0.181811& 1.7789& 0.524270& 0.9094\\
											
											0.088388&2945&    0.071856&  1.7999&   0.053249& 1.7716 & 0.273319& 0.9397\\
											
											0.044194& 12033&    0.022050&  1.7044& 0.017351& 1.6178&   0.143736&   0.9272
											
											\\
											 0.022097&48641&0.007491& 1.5575&0.006560& 1.4033& 0.077744& 0.8866\\
												\hline				
										\end{tabular}
										\begin{tabular}{|c|c||c | c ||c|c|| c|c||c|c||c|c||}
											\hline
											$h$ &	$\textbf{nu}$&$\err_\disc(\bv)$ & Order  &$\err_\disc(\nabla \bv)$ & Order  &$\err_\disc(\hessian\bv)$ & Order  \\ 
											\hline\\[-10pt]  &&\\[-10pt]
											0.707107& 33&1.910146& -& 1.293881& -& 1.351562& -\\
											
											0.353553& 161&  0.794724& 1.2652&  0.569137& 1.1849&  0.966468&  0.4838\\
											
											0.176777&705& 0.229244& 1.7936& 0.167686& 1.7630& 0.527682&0.8731\\
											
											0.088388&2945&0.064624&1.8267& 0.047896& 1.8078&   0.275565& 0.9373\\
											
											0.044194& 12033&   0.019339& 1.7406& 0.015209&  1.6550&   0.144849&   0.9278
											\\
											0.022097&48641& 0.006411&1.5929& 0.005694&1.4175& 0.078259& 0.8882\\
												\hline				
										\end{tabular}
									\end{center}}	}\label{table4}
								\end{table}
								Since $\O$ is non-convex, we expect only sub-optimal order of convergences in the energy, $H^1$ and $L^2$ norms. Table \ref{table4} confirms these estimates numerically. 

\section{Conclusion}\label{sec:conclusion}

This paper focuses on the HDM for an abstract semilinear fourth order system of partial differential equations. The abstract model covers in particular the Navier--Stokes equations in stream function vorticity formulation, and the von K\'{a}rm\'{a}n equations. The HDM gives a generic analysis framework that simultaneously covers several numerical schemes, such as conforming and non-conforming finite elements, as well as schemes based on GR approaches. The convergence of the scheme is established, using compactness techniques and without assuming any smoothness or particular structure of the continuous solution. Numerical results on both Navier--Stokes and von K\'{a}rm\'{a}n equations, using the Morley FEM or the GR approach that are not available in literature, are presented. 

The convergence analysis via error estimates is also proved in the HDM framework (see \cite[Section 3.5.2]{DSthesis} for details) when this solution is regular, and a companion operator, that maps the discrete space to the continuous space, can be designed with approximation properties. This error estimates provide a linear order of convergence for conforming FEMs in the energy norm. The construction of an appropriate companion operator for the GR method is an ongoing work.

\medskip

\textbf{Acknowledgments: }The work of the first author was partially supported by the Australian Government through the Australian Research Council’s Discovery Projects funding scheme (project number DP170100605). The second author was supported by DST SERB MATRICS grant MTR/2017/000199.	
\bibliographystyle{abbrv}
\bibliography{Fourth_order_elliptic}


\appendix
\section{Proof of the properties associated with ncFEMs and GR Method}\label{appen.prop}
This section deals with the proofs of the properties associated with the ncFEMs and GR methods stated in Theorems \ref{th.MorleyAdini} and \ref{thm:HD.GR}.
\subsection{\bf Non-conforming FEM}
The preliminaries below are useful to prove the convergence of the Adini and Morley HDM for non-linear equations.

Let $h_\sigma$ denote the diameter of $\sigma$ and for all $w \in H^1(\mesh)$ (the broken Sobolev space $H^1$ on the mesh),
define  $\norm{{\cdot}}{dG,\mesh}$ by
$$
\norm{w}{dG,\mesh}^2:=\norm{\nabla_\mesh w}{}^2 + \sum_{\edge \in \edges}^{}\frac{1}{h_\edge}\norm{\llbracket w \rrbracket}{L^2(\edge)}^2.
$$


\begin{lemma}\cite[Theorems 5.3, 5.6]{DG_DA}\label{lemma.dG} Let $\ell\ge 0$ be a non-negative integer. It holds
	\begin{itemize}
		\item[(i)][Discrete Sobolev embedding] For all $v_h \in \mathbb{P}_\ell(\mesh)$, $\norm{v_h}{L^4} \lesssim  \norm{v_h}{dG,\mesh}.$
		\item [(ii)][Discrete Rellich theorem] Let $(\mesh_{h_m})_{m\in\N}$ be sequence of regular triangular or rectangular meshes, whose diameter $h_m$ tend to $0$ as $m\to\infty$. For all $m\in\N$, let $v_m\in\mathbb{P}_\ell(\mesh_{h_m})$. If $(\norm{v_m}{dG,\mesh_{h_m}})_{m\in\N}$ is bounded, then, for all $1 \le q < 2^*$ (where $2^*$ is a Sobolev exponent of $2$), the sequence $(v_m)_{m\in\N}$ is relatively compact in $L^q(\O)$.
		\item[(iii)] \cite[Lemma 6.2]{DS_HDM}
		Let $w\in H^1(\mesh)$. If for all $\edge \in \edges$ there exists $x_{\edge} \in \edge$ such that $\llbracket w \rrbracket (x_{\edge})=0$, then $\norm{w}{dG,\mesh} \lesssim \norm{\nabla_{\mesh}w}{}.$
	\end{itemize}
\end{lemma}

\begin{lemma}\cite{HDM_linear,DS_HDM}\label{lemma.ncFEM}
	Let $\disc$ be a HD for the Morley (resp.\ Adini) ncFEM in the sense of Definition \ref{def.HD.Morley} (resp. Definition \ref{def.HD.Adini}). Then,
	\begin{align*}
	&\forall  v_\disc \in X_{\disc,0}, \;\norm{\Pi_\disc v_\disc}{} \lesssim \norm{\hd v_\disc}{} ,\quad \forall \xi \in H^2(\O;\R^{2 \times 2}),\, W_{\disc}(\xi) \lesssim h \norm{\xi}{H^2(\O;\R^{2 \times 2})}.
	\end{align*}
\end{lemma}
We now prove the four properties associated with the Morley and Adini ncFEMs.

\begin{proof}[Proof of Theorem \ref{th.MorleyAdini}]	
%
	\textsc{(i) The Morley element.}
	
	{\it (i)} Let $v_\disc \in X_{\disc,0}$. Since $\llbracket \nabla_\disc v_\disc \rrbracket=0$ at the midpoints of the edges, Lemma \ref{lemma.dG}$(i)$ and Lemma \ref{lemma.dG}$(iii)$ lead to
	\be \label{coercive_hessian}
	\norm{\nabla_\disc v_\disc}{L^4} \lesssim \norm{\nabla_\disc v_\disc}{dG,\mesh}  \lesssim  \norm{\hessian_{\disc}v_\disc}{}.\ee
	The combination of this and Lemma \ref{lemma.ncFEM}$(i)$ prove $C_\disc \lesssim 1$.
	
	{\it (ii)} For $\varphi\in H^{3}(\O)\cap H^2_0(\O)$, the standard interpolant \cite{PLPL75,ciarlet1978finite} satisfies
	\begin{align} \label{Pd.nonconforming}
	&\norm{I_h\varphi-\varphi}{}\lesssim h^{3} \norm{\varphi }{H^3(\O)}, \;\norm{\nabla_\mesh I_h\varphi-\nabla \varphi}{L^4} \lesssim  h^{\frac{3}{2}} \norm{\varphi}{H^3(\O)},\nonumber \\
	&\mbox{ and }\norm{\hessian_\mesh I_h\varphi-\hessian\varphi}{} \lesssim  h \norm{\varphi }{H^3(\O)}.
	\end{align}
	Hence, selecting $w \in X_{\disc,0}$ corresponding to the degrees of freedom of $I_h\varphi$ in the definition of $S_\disc(\varphi)$ yields the result on consistency.

	{\it (iii)} For all $\cell \in \mesh$ and for any $\edge \in \edgescv$, denote the unit vector normal to $\edge$ outward $\cell$ by $n_{\cell,\edge}$. An integration by parts shows
	\begin{align} \label{limitconformitygrad}
	\bigg|	\int_\O \Big(\nabla_\disc v_\disc \cdot \phi + \Pi_\disc v_\disc \mbox{div} \phi \Big)\d\x \bigg|&  =\bigg|\sum_{\cell \in \mesh}^{}\sum_{\edge \in \edgescv}^{} \int_{\edge}^{}(\phi \cdot n_{\cell,\edge})\Pi_\disc v_\disc \d s(\x)\bigg|.
	\end{align}
	The proof of the result on $\hat{W}_{\disc}$ then follows from a slightly modified version of \cite[Lemma 3.5]{PLPL75}. However, for the sake of completeness, we provide a proof.
	Let $V_1$ be the space of all globally continuous piecewise linear functions and let $\Pi_1:V_h \rightarrow V_1$ be the interpolation operator such that $\Pi_1v_h$ equal to $v_h$ at the vertices of all triangle $K$, $v_h \in V_h$. Then 
	\be \label{limitconformity22}
	\sum_{\cell \in \mesh}^{}\sum_{\edge \in \edgescv}^{} \int_{\edge}^{}(\phi \cdot n_{\cell,\edge})\Pi_\disc v_\disc \d s(\x)=\sum_{\cell \in \mesh}^{}\sum_{\edge \in \edgescv}^{} \int_{\edge}^{}P(v_h-\Pi_1{v_h})\d s(\x),	
	\ee 
	where $P=\phi \cdot n_{\cell,\edge} $ and $v_h=\Pi_\disc v_\disc$. Note that a change of variables yields
	\be \label{limitconformityrefPi1}
	\int_{\edge\in \edgescv}^{}P(v_h-\Pi_1{v_h})\d s(\x)=|\edge|\int_{\hat{\edge} \in \edges_{\hat{\cell}}}\hat{P}(\hat{v_h}-\hat{\Pi_1}{\hat{v_h}})\d s(\x),
	\ee
	where ${\hat{\cell}}$ is the reference finite element. The continuous trace inequality \cite{DG_DA}, the discrete trace inequality \cite[Lemma 1.46]{DG_DA}, an interpolation estimate \cite{ciarlet1978finite} and Young's inequality show
	$$\bigg|\int_{\hat{\edge} \in \edges_{\hat{\cell}}}\hat{P}(\hat{v_h}-\Pi_1{\hat{v_h}})\d s(\x)\bigg|\lesssim \big(|{\hat{P}}|_{0,\hat{K}}+|{\hat{P}}|_{1,\hat{K}}\big)|{\hat{v}}|_{2,\hat{K}},$$ 
	where $|\cdot|_{\ell,\hat{K}}$ denotes the seminorm on $H^\ell(\hat{K})$, $\ell\ge 0$. Substitute the above estimate in \eqref{limitconformityrefPi1}, transform from $\hat{K}$ to $K$ using $|\hat{\psi}|_{\ell,\hat{K}} \lesssim h^{\ell-1}|\psi|_{\ell,K}$ for $\psi \in H^\ell(K)$ (\cite[Theorem 3.1.2]{ciarlet1978finite}), sum over all the edges and then use \eqref{limitconformity22} to obtain
	\begin{align}
	\bigg|\sum_{\cell \in \mesh}^{}\sum_{\edge \in \edgescv}^{} \int_{\edge}^{}(\phi &\cdot n_{\cell,\edge})\Pi_\disc v_\disc \d s(\x)\bigg|\lesssim (h \norm{\phi}{}+h^2 \norm{\nabla\phi}{} )\norm{\hd v_\disc}{}.\nonumber 
	\end{align}
	This in \eqref{limitconformitygrad} reads
	$$\left|	\int_\O \Big(\nabla_\disc v_\disc \cdot \phi + \Pi_\disc v_\disc \mbox{div} \phi \Big)\d\x \right|\lesssim (h \norm{\phi}{}+h^2 \norm{\nabla\phi}{})\norm{\hd v_\disc}{}.$$
	The above estimate with Lemma \ref{lemma.ncFEM} leads to the required estimate on $\hat{W}_{\disc}$ and $W_\disc$ and thus establish limit-conformity.

	{\it (iv)} Let a sequence $u_m \in X_{\disc_m,0}$ be such that $(\norm{u_m}{\disc_m})_{m \in \N}$ is bounded. Since $\llbracket \Pi_{\disc_m} u_m \rrbracket$ = 0 at the edge vertices, Lemma \ref{lemma.dG}$(iii)$ and estimate \eqref{coercive_hessian} yield
	$ \norm{\Pi_{\dm} u_m}{dG,\mesh_m} \lesssim \norm{\nabla_{\dm} u_m}{} \lesssim \norm{\nabla_{\dm} u_m}{L^4}\lesssim  \norm{\hessian_{\dm}u_m}{}.$  At the edge midpoints $\llbracket \nabla_{\dm} u_m \rrbracket=0$ and consequently, Lemma \ref{lemma.dG}$(iii)$ with $w=\nabla_{\dm} u_m$ shows
	$ \norm{\nabla_{\dm} u_m}{dG,\mesh_m}$ $\lesssim \norm{\hessian_{\dm}u_m}{}.$
	Use the fact that $(\norm{u_m}{\disc_m})_{m \in \N}$ is bounded to deduce $(\Pi_{\disc_m}u_m)_{m \in \N}$ and $(\nabla_{\disc_m}u_m)_{m \in \N}$ are bounded in the $\norm{\cdot}{dG,\mesh_m}$ norm. Lemma \ref{lemma.dG}$(ii)$ then gives the relative compactness of $(\Pi_{\disc_m}u_m)_{m \in \N}$ in $L^2(\O)$, and of $(\nabla_{\disc_m}u_m)_{m \in \N}$ in $L^4(\O;\R^d)$.
	
	\medskip
	
	\textsc{(ii) The Adini element.}
	
	{\it (i)} Since $\nabla_\disc v_\disc$ is continuous at the vertices of elements in $\mesh$ and $\nabla_\disc v_\disc$ vanish at vertices along $\partial \O$, $\llbracket \nabla_\disc v_\disc \rrbracket=0$ at the vertices. As a consequence, Lemma \ref{lemma.dG}$(i)$-$(iii)$ leads to
	$	\norm{\nabla_\disc v_\disc}{L^4} \lesssim \norm{\nabla_\disc v_\disc}{dG,\mesh}  \lesssim  \norm{\hessian_{\disc}v_\disc}{}.$ 
This and Lemma \ref{lemma.ncFEM} concludes that $C_\disc {\lesssim 1}$.
	
	\noindent {\it (ii)} The standard interpolant satisfies \eqref{Pd.nonconforming} and hence yields the desired estimate on the consistency measure $S_\disc$.
	
	\noindent {\it (iii)} Any $\edge \in \edges$ is associated to a fixed orientation of the unit normal vector $n_\edge$ on $\edge$. For $K \in \mesh$ and $\edge\in\edges_K$, set $n_\edge:=n_{\cell,\edge}$. Apply integration by parts in each cell to obtain
	$$\bigg|	\int_\O \Big(\nabla_\disc v_\disc \cdot \phi + \Pi_\disc v_\disc \mbox{div} \phi \Big)\d\x \bigg|=\bigg|\sum_{\edge \in \edges}^{} \int_{\edge}^{}(\phi \cdot n_\sigma)\llbracket \Pi_\disc v_\disc \rrbracket \d s(\x)\bigg|.$$
	Here, the jump has been chosen in a compatible way with the orientation determined by $n_\sigma$. Since $\Pi_\disc v_\disc \in H^1_0(\O)\cap C(\overline{\O})$, $\llbracket \Pi_\disc v_\disc \rrbracket=0$, which implies $\hat{W}_{\disc}(\phi)=0$.
	This and Lemma \ref{lemma.ncFEM} yields an estimate on the limit-conformity measures, namely $\hat{W}_{\disc}$ and $W_\disc$.
	
	{\it (iv)} The compactness proof follows as for the Morley element using the fact that $\llbracket \Pi_\disc v_\disc \rrbracket=0$ and $\llbracket \nabla_\disc v_\disc \rrbracket=0$ at the vertices.
\end{proof}
\subsection{\bf Method based on GR operators}
Recall the properties $\propty{0}-\propty{5}$ method based on GR operators associated with Section \ref{sec.grmethod}. 
\begin{lemma} \cite[Theorem 4.3]{HDM_linear}\label{lemma.GR}
	Let $\disc$ be a HD in the sense of Definition \ref{def:HDM:GR} and $(V_h,I_h,Q_h,\stab_h)$ satisfying \propty{0}--\propty{5}. Let $v \in X_{\disc,0}$. Then 
	\begin{itemize}
		\item[(i)] 	$\norm{\nabla (Q_h \nabla v)}{}+\norm{Q_h\nabla v-\nabla v}{}\le \sqrt{2}\norm{\hd v}{}$,
		\item[(ii)]   $\norm{\Pi_\disc v}{} \lesssim  \norm{\hd v}{},$
		\item[(iii)] $\forall \phi \in {H^2}(\O)$, $\norm{\Pi_\disc I_h\phi-\phi}{}\lesssim h\norm{\phi}{{H^2}(\O)},$
		\item[(iv)] $\forall \varphi \in W$, $\norm{\nabla (Q_h \nabla  I_h\varphi)-\nabla\nabla \varphi}{}\lesssim h \norm{\varphi}{W}$ and  $\norm{\hd I_h\varphi-\hessian \varphi}{}\lesssim h\norm{\varphi}{W},$
		\item[(v)] $\forall \:\xi \in H^2(\O;\R^{d \times d})$, $W_{\disc}(\xi) \lesssim h \norm{\xi}{H^2(\O;\R^{d \times d})}.$
	\end{itemize}	
\end{lemma}

We can now prove the accuracy measures for the GR methods.

\medskip

\begin{proof}[Proof of Theorem \ref{thm:HD.GR}]~
	$(i)$ For $v\in X_{\disc,0}$, since $\nabla_\disc v\in V_h^d\subset H^1_0(\O;\R^d)$, the Sobolev embedding $H^1(\O) \hookrightarrow L^4(\O)$ and Lemma \ref{lemma.GR}$(i)$ yield
	$\norm{\nabla_\disc v}{L^4}=\norm{Q_h\nabla v}{L^4}$ $\lesssim \norm{\nabla(Q_h\nabla v)}{}
	\le 	\sqrt{2}\norm{\hd v}{}.$
	This estimate along with Lemma \ref{lemma.GR}$(ii)$ show that the coercivity measure $C_\disc\lesssim 1$.
	
	\smallskip
	
	$(ii)$ Let $\varphi \in W\subset H^3(\O)\cap H^2_0(\O)$ and  choose $v=I_h\varphi \in X_{\disc,0}$. A use of Sobolev embedding $H^1(\O) \hookrightarrow L^4(\O)$ and Lemma \ref{lemma.GR}$(iv)$ leads to
	\be \label{consistency_gradrec2}
	\norm{\nabla_\disc v-\nabla \varphi}{L^4}\lesssim \norm{\nabla (Q_h \nabla v)-\nabla\nabla \varphi}{}\lesssim h \norm{\varphi}{W}.
	\ee
	Thus, the consistency estimate on $S_\disc(\varphi)$ follows from \eqref{consistency_gradrec2} and Lemma \ref{lemma.GR}$(iii)$-$(iv)$.
	
	\smallskip
	
	\noindent $(iii)$ For $ \xi \in H^2(\O;\R^{d\times d})$, we have $W_{\disc}(\xi) \lesssim h \norm{\xi}{H^2(\O;\R^{d \times d})}$ from Lemma \ref{lemma.GR}$(v)$. Since $\Pi_\disc v \in H^1_0(\O)$ for all $v\in X_{\disc,0}$, an integration-by-parts shows that $\hat{W}_\disc\equiv 0$.
	
	$(iv)$ Let a sequence $u_m \in X_{\disc_m,0}$ be such that $(\norm{u_m}{\disc_m})_{m \in \N}$ is bounded. Since $\Pi_{\disc_m}u_m\in H^1_0(\Omega)$ and $Q_{h_m} \nabla u_m\in H^1_0(\Omega)^d$, a use of triangle inequality, the Poincar\'e inequality and Lemma \ref{lemma.GR}$(i)$ leads to
	\begin{align*}
	\norm{\nabla(\Pi_{\dm} u_m)}{}&=\norm{\nabla u_m}{}\le \norm{Q_{h_m} \nabla u_m}{}+\norm{Q_{h_m} \nabla u_m-\nabla u_m}{}\\
	&\lesssim \norm{\nabla Q_{h_m} \nabla u_m}{}+\norm{Q_{h_m} \nabla u_m-\nabla u_m}{}\lesssim \norm{\hessian_{\dm}u_m}{}.
	\end{align*}
	Since $(\norm{u_m}{\disc_m})_{m \in \N}$ is bounded, it follows that $(\nabla(\Pi_{\disc_m}u_m))_{m \in \N}$ is bounded in $L^2(\O;\R^d)$ and hence the standard Rellich theorem shows that $(\Pi_{\dm} u_m)_{m \in \N}$ is relatively compact in $L^2(\O)$. Note that $\nabla_{\disc_m}u_m=Q_{h_m} \nabla u_m\in H^1_0(\Omega)^d$. From Lemma \ref{lemma.GR}$(i)$, $\norm{\nabla Q_{h_m} \nabla u_m}{} \le C\norm{\hessian_{\dm}u_m}{}$. Thus, the Rellich and Sobolev imbedding theorems yield the required compactness property of $(\nabla_{\disc_m}u_m)_{m \in \N}$ in $L^4(\O;\R^d)$.
\end{proof}

\end{document}